\documentclass[12pt]{amsart}\pdfoutput=1
\usepackage{amssymb,amsmath,amsthm,amsbsy,amsfonts,mathtools, floatflt}
\usepackage{color,xcolor,overpic,tikz,graphicx}
\usepackage{bbm,enumitem,esvect}
\usepackage[all]{xy}
\usepackage[english]{babel}
\usepackage[pdfpagelabels]{hyperref}
\usepackage{ucs}
\usepackage[utf8x]{inputenc}

\DeclareGraphicsExtensions{.pdf,.jpg}

\theoremstyle{plain}
\newtheorem{mainthm}{Theorem}
\newtheorem{thm}{Theorem}[section]\newtheorem*{thm*}{Theorem}
\newtheorem{lem}[thm]{Lemma}\newtheorem{lemma}[thm]{Lemma}
\newtheorem{corollary}[thm]{Corollary}
\newtheorem{prop}[thm]{Proposition}

\theoremstyle{definition}
\newtheorem{defn}[thm]{Definition}
\newtheorem{example}[thm]{Example}

\theoremstyle{remark}
\newtheorem{rem}[thm]{Remark}\newtheorem{remark}[thm]{Remark}

\newcommand{\lr}{\ell_R}

\newcommand{\one}{\mathbf{1}}
\newcommand{\zero}{\mathbf{0}}

\newcommand{\R}{\mathbb{R}}
\newcommand{\Z}{\mathbb{Z}}

\newcommand{\size}[1]{\ensuremath{\vert #1 \vert}}
\newcommand{\onto}{\twoheadrightarrow}
\newcommand{\into}{\hookrightarrow}

\newcommand{\id}{\mathrm{1}} 

\newcommand{\nll}{\nu}
\newcommand{\ncplx}{\textsc{NullCplx}}
\newcommand{\nprt}{\Pi^0}
\newcommand{\spn}{\textsc{Span}}

\newcommand{\arr}{\textsc{Arr}}
\newcommand{\prt}{\Pi}
\newcommand{\cyc}{\textsc{cyc}}

\newcommand{\Pos}{\textsc{Pos}}
\newcommand{\Neg}{\textsc{Neg}}
\newcommand{\Basic}{\textsc{Basic}}

\newcommand{\mov}{\textsc{Mov}}
\newcommand{\fix}{\textsc{Fix}}
\newcommand{\im}{\textsc{Im}}
\newcommand{\Ker}{\textsc{Ker}}

\newcommand{\Sym}{\mathfrak{S}}
\newcommand{\affS}{\widetilde{\mathfrak{S}}}

\tikzstyle{Square}=[inner sep=2, draw=black,fill=white]
\tikzstyle{Rounded}=[inner sep=2, draw=black,fill=white, rounded corners]
\tikzstyle{MiniRounded}=[Rounded,midway,scale=.75]
\tikzstyle{Circle}=[draw=black, fill=white,circle, inner sep=1]
\tikzstyle{Poly}=[thick,color=green!20!black,fill=green!30,join=bevel]

\begin{document}

\title[Computing affine reflection length]{Computing
  reflection length in\\ an affine Coxeter group}

\author[J. B. Lewis]{Joel Brewster Lewis}
\thanks{Work of Lewis supported by NSF grant DMS-1401792}

\author[J. McCammond]{Jon McCammond}

\author[T. K. Petersen]{T. Kyle Petersen}
\thanks{Work of Petersen supported by a Simons Foundation
  collaboration travel grant.}

\author[P. Schwer]{Petra Schwer}
\thanks{Work of Schwer supported by DFG grant SCHW 1550 4-1 within SPP
  2026} 


\date{\today}

\begin{abstract}
  In any Coxeter group, the conjugates of elements in its Coxeter
  generating set are called reflections and the reflection length of
  an element is its length with respect to this expanded generating
  set.  In this article we give a simple formula that computes the
  reflection length of any element in any affine Coxeter group and we
  provide a simple uniform proof.
\end{abstract}

\maketitle
\tableofcontents

\section*{Introduction}

In any Coxeter group $W$, the conjugates of elements in its standard
Coxeter generating set are called \emph{reflections} and we write $R$
for the set of all reflections in $W$.  The reflections generate $W$
and the associated \emph{reflection length} function $\lr\colon W \to \Z_{\geq 0}$
records the length of $w$ with respect to this expanded generating set.
When $W$ is spherical, i.e., finite, reflection length can be given an
intrinsic, geometric definition, as follows. Define a \emph{root subspace} to be
any space spanned by a subset of the corresponding root system and define the
\emph{dimension} $\dim(w)$ of an element $w$ in $W$ to be the minimum dimension
of a root subspace that contains the move-set of $w$.  (See
Section~\ref{sec:dim} for the details.)

\begin{thm*}[{\cite[Lem.~2]{Carter72}}]
Let $W$ be a spherical Coxeter group and let $w \in W$.  Then $\lr(w) = \dim(w)$.
\end{thm*}

When $W$ is infinite, much less is known \cite{McPe-11, Duszenko12,
  MST-affine-deligne-lusztig}, even in the affine case.  In this
article, we give a simple, analogous formula that computes the
reflection length of any element in any affine Coxeter group and we
provide a simple uniform proof.

\newcommand{\maindim}{0}
\begin{mainthm}[Formula]\label{main:dim}
  Let $W$ be an affine Coxeter group and let $p\colon W \onto W_0$ be
  the projection onto its associated spherical Coxeter group.  For any
  element $w\in W$, its reflection length is 
  \[\lr(w) = 2 \cdot \dim(w) - \dim(p(w)) = 2d + e,\] 
  where $e = \dim(p(w))$ and $d = \dim(w)-\dim(p(w))$.
\end{mainthm}

We call $e = e(w) = \dim(p(w))$ the \emph{elliptic dimension} of $w$
and we call $d = d(w) = \dim(w) - \dim(p(w))$ the \emph{differential
  dimension} of $w$.  Both statistics are geometrically meaningful.
For example, $e(w) = 0$ if and only if $w$ is a translation (that is,
if it sends every point $x$ to $x+\lambda$ for some fixed vector
$\lambda$), and $d(w)=0$ if and only if $w$ is elliptic (that is, if
it fixes a point) -- see Proposition~\ref{prop:stat-geo}.

The translation and elliptic versions of our formula were already
known (see \cite{McPe-11} and \cite{Carter72}, respectively) and these
special cases suggest that the formula in Theorem~\ref{main:dim} might
correspond to a carefully chosen factorization of $w$ as a product of
a translation and an elliptic element.  This is indeed the case.

\newcommand{\mainfact}{1}
\begin{mainthm}[Factorization]\label{main:fact}
  Let $W$ be an affine Coxeter group.  For every element $w\in W$, there
  is a translation-elliptic factorization $w = t_\lambda  u$ such that
  $\lr(t_\lambda) = 2d(w)$ and $\lr(u) = e(w)$.  
  In particular, $\lr(w) = \lr(t_\lambda) + \lr(u)$ for
  this factorization of $w$.
\end{mainthm}

The proof of Theorem~\ref{main:fact} relies on a nontrivial technical
result recently established by Vic Reiner and the first author
\cite[Corollary~1.4]{LewisReiner-16}. Elements in an affine Coxeter
group typically have many different potential translation-elliptic
factorizations and the most common way to find one is to view the
group as a semidirect product.  For an affine Coxeter group $W$ that
naturally acts cocompactly on an $n$-dimensional euclidean space, the
set of translations in $W$ forms a normal abelian subgroup $T$
isomorphic to $\Z^n$, and the quotient $W_0 = W/T$ is the spherical
Coxeter group associated with $W$.  An identification of $W$ as a
semidirect product $T \rtimes W_0$ corresponds to a choice of an
inclusion map $i\colon W_0 \into W$ that is a section of the
projection map $p\colon W \onto W_0$.  There is a unique point $x$
fixed by the subgroup $i(W_0)$ and every element has a unique
factorization $w = t_\lambda u$ where $t_\lambda$ is a translation in
$T$ and $u$ is an elliptic element in the copy of $W_0$ that fixes
$x$.  However, for some elements $w$ none of the
translation-elliptic factorizations that come from an identification
of $W$ as a semidirect product satisfy Theorem~\ref{main:fact} -- see
Example~\ref{ex:insuff}.

In the particular case of the symmetric group $\Sym_n$ (the spherical
Coxeter group of type A), reflection length also has a natural
combinatorial characterization: $\lr(w) = n - c(w)$, where $c(w)$ is
the number of \emph{cycles} of the permutation $w$ \cite{Denes59}.
When $W$ is the affine symmetric group $\affS_n$, we show that the
differential and elliptic dimensions of an element can also be given a
combinatorial interpretation, leading to the very similar formula in
Theorem \ref{thm:lr-affS}.

\subsection*{Structure}
The article has five sections. In Section \ref{sec:dimensions}, we
discuss the necessary background on spherical and affine Coxeter
groups, including the definition of the dimension of an element. In
Section \ref{sec:mains}, we use the ideas of Section
\ref{sec:dimensions} to prove our main theorems. In Section
\ref{sec:gen-fns}, we discuss our understanding of the local
distribution of reflection length. In Section \ref{sec:aff-sym}, we
restrict our attention to the affine symmetric groups and develop a
combinatorial understanding of reflection length. In Section
\ref{sec:future}, we point to topics for further study. Finally, in
Appendix \ref{sec:null-comp}, we describe an algorithm for computing
reflection length in the affine symmetric group.

\section{Dimensions in affine Coxeter groups}\label{sec:dimensions}
In this section we recall the relevant background on reflection length
and basic notions for spherical and affine Coxeter groups. We then
develop a notion of the dimension of an element in an affine Coxeter
group. We assume the reader is familiar with Coxeter groups at the
level of Humphreys \cite{Humphreys-90}.

\subsection{Reflection length}\label{sec:lr}

This section reviews basic facts about reflection length in arbitrary
Coxeter groups.  

\begin{defn}[Reflection length]\label{def:lr}
  Given a Coxeter group $W$ with standard Coxeter
  generating set $S$, a \emph{reflection} is any element conjugate
  to an element in $S$.  Let $R$ denote the set of all reflections in
  $W$; note that $R$ is infinite whenever the
  Coxeter group $W$ is infinite.  The \emph{reflection length} $\lr(w)$
  of an element $w$ is the minimum integer $k$ such that
  there exist reflections $r_i \in R$ with $w = r_1 r_2 \cdots r_k$.
  The identity has reflection length $0$.
\end{defn}

\begin{rem}[Prior results]\label{rem:prior}
  When $W$ is a spherical Coxeter group, Carter's lemma mentioned in
  the introduction establishes that the reflection length of an
  element $w \in W$ may be computed in terms of the dimension of its
  move-set or fixed space (see Section~\ref{sec:finite Coxeter groups}
  below for the definitions): $\lr(w) = \dim (\mov(w)) = \dim(V) -
  \dim (\fix(w))$. When $W$ is an affine Coxeter group acting
  cocompactly on an $n$-dimensional euclidean space, the second and
  third authors have shown that $\lr(w)$ is bounded above by $2n$ for
  every element $w$ in $W$, and established an exact formula for the
  reflection length of a translation \cite[Theorem~A and
    Proposition~4.3]{McPe-11}.  Mili\'{c}evi\'{c}, Thomas and the
  fourth author found better bounds and explicit formulas in some
  special cases as part of their work on affine Deligne--Lusztig
  varieties \cite[Theorem~11.3 and
    Corollary~11.5]{MST-affine-deligne-lusztig}.  Finally, when $W$ is
  neither spherical nor affine, Duszenko has shown that the reflection
  length function is always unbounded \cite{Duszenko12}.
\end{rem}

\begin{rem}[Basic properties]\label{rem:lr-basic}
  If $W$ is a reducible Coxeter group $W = W_1 \times W_2$, the
  reflection length of an element $w \in W$ is just the sum of the
  reflection lengths of its factors \cite[Proposition~1.2]{McPe-11}.
  Thus, it is sufficient to study \emph{irreducible} Coxeter groups,
  for which there is a well-known classification \cite[Section
    6.1]{Humphreys-90}.  Since the set $R$ of reflections is a union
  of conjugacy classes, the length function $\lr \colon W \to \Z_{\geq
    0}$ is constant on conjugacy classes.  This implies, in
  particular, that $\lr(uv) = \lr(vu)$.  The function $\lr$ can also
  be viewed as the natural combinatorial distance function on the
  Cayley graph of $W$ with respect to the generating set $R$, which
  means that $\lr$ satisfies the triangle inequality $\lr(u)-\lr(v)
  \leq \lr(uv) \leq \lr(u)+\lr(v)$.  Finally, for every Coxeter group
  there is a homomorphism onto $\Z/2\Z$ that sends every reflection to
  the non-identity element.  As a consequence, reflection length has a
  parity restriction: $\lr(uv) = \lr(u) + \lr(v) \mod 2$.  These facts
  combine to show that $\lr(rw)=\lr(wr)=\lr(w)\pm 1$ for every element
  $w \in W$ and reflection $r\in R$.
\end{rem}

One standard result about reflection factorizations is that they can
be rewritten in many different ways.

\begin{lemma}[Rewriting reflection factorizations]\label{lem:hurwitz}
  Let $w = r_1 r_2 \cdots r_k$ be a reflection factorization of an
  element $w$ of a Coxeter group.  For any selection $1 \leq i_1 < i_2
  < \cdots < i_m \leq k$ of positions, there is a length-$k$
  reflection factorization of $w$ whose first $m$ reflections are
  $r_{i_1} r_{i_2} \cdots r_{i_m}$ and another length-$k$ reflection
  factorization of $w$ where these are the last $m$ reflections.
\end{lemma}

\begin{proof}
  Because the set of reflections is closed under conjugation, we have
  for any reflections $r$ and $r'$ that the elements $r'' = r r' r$
  and $r''' = r' r r'$ are also reflections, satisfying $r r' = r'' r
  = r' r'''$.  Thus in any reflection factorization one may replace a
  consecutive pair $rr'$ of factors with the pair $r''r$ (moving $r$
  to the right one position) or $r'r'''$ (moving $r'$ to the left one
  position) without changing the length of the factorization.
  Iterating these rewriting operations allows one to move any subset
  of reflections into the desired positions.
\end{proof}

\begin{rem}[Hurwitz action]
  The individual moves in the proof of Lemma~\ref{lem:hurwitz} are
  called \emph{Hurwitz moves}.  Globally, they correspond to an
  action, called the \emph{Hurwitz action}, of the $k$-strand braid
  group on the set of all length-$k$ reflection factorizations of a
  given word $w$.
\end{rem}

\subsection{Spherical Coxeter groups}\label{sec:finite Coxeter groups}

In this section, we discuss the spherical Coxeter groups and their
connection to root systems.

\begin{defn}[Spherical Coxeter groups]\label{def:spherical coxeter group}
  A \emph{euclidean vector space} $V$ is an $n$-dimensional real
  vector space equipped with a positive definite inner product
  $\langle \cdot, \cdot \rangle$.  A \emph{crystallographic root
    system $\Phi$} is a finite collection of vectors that span a real
  euclidean vector space $V$ satisfying a few elementary properties --
  see \cite{Humphreys-90} for a precise definition. (While there are
  non-crystallographic root systems as well, it is only the
  crystallographic root systems that arise in the study of affine
  Coxeter groups.) The elements of $\Phi$ are called \emph{roots}.
  Each crystallographic root system corresponds to a finite (or
  \emph{spherical}) Coxeter group $W_0$, as follows: for each $\alpha$
  in $\Phi$, $H_\alpha$ is the hyperplane through the origin in $V$
  orthogonal to $\alpha$, and the unique nontrivial isometry of $V$
  that fixes $H_{\alpha}$ pointwise is a \emph{reflection} that we
  call $r_{\alpha}$.  The collection $R = \{r_{\alpha} \mid \alpha \in
  \Phi\}$ generates the spherical Coxeter group $W_0$, and $R$ is its
  set of reflections in the sense of Definition~\ref{def:lr}. In
  Figure \ref{fig:hyperplanes-finite} we see the hyperplanes for each
  irreducible crystallographic root system of rank two.
\end{defn}

\begin{figure}
  \begin{tabular}{c c c}
    \begin{tikzpicture}[cm={1,0,.5,.8660254,(0,0)}, >=stealth,baseline=2.6cm]
      \draw[very thick] (0,3)--(4,3);
      \draw[very thick] (2,1)--(2,5);
      \draw[very thick] (0,5)--(4,1);
    \end{tikzpicture}
    &
    \begin{tikzpicture}[>=stealth,baseline=0,scale=.85]
      \draw[very thick] (-2,2)--(2,-2);
      \draw[very thick] (-2,-2)--(2,2);
      \draw[very thick] (0,-2)--(0,2);
      \draw[very thick] (-2,0)--(2,0);
    \end{tikzpicture}
    &
    \begin{tikzpicture}[cm={1,0,.5,.8660254,(0,0)}, >=stealth,baseline=2.6cm]
      \draw[very thick] (0,3)--(4,3);
      \draw[very thick] (2,1)--(2,5);
      \draw[very thick] (0,5)--(4,1);
      \draw[very thick] (1,2)--(3,4);
      \draw[very thick] (0,4)--(4,2);
      \draw[very thick] (1,5)--(3,1);
    \end{tikzpicture}
    \\
    (a) & (b) & (c)
  \end{tabular}
  \caption{ Reflecting hyperplanes associated to the root systems of
    types (a) $A_2$, (b) $B_2$ and $C_2$, and (c) $G_2$.}
  \label{fig:hyperplanes-finite}
\end{figure}
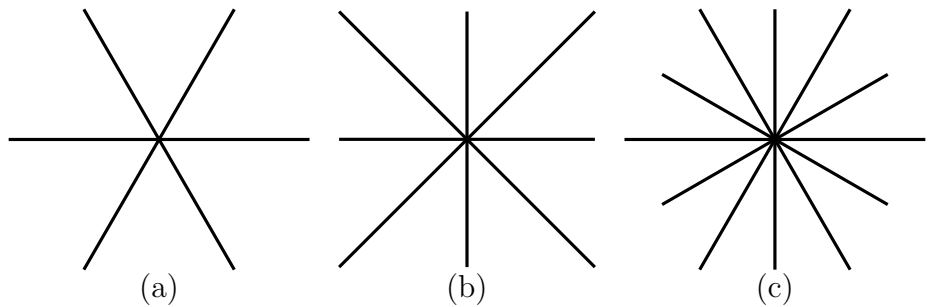

In a spherical Coxeter group constructed as above, each element is
associated to two fundamental subspaces, its fixed space and its
move-set.

\begin{defn}[Fixed space and move-set]\label{def:linear-fixed-space}
  Given an orthogonal transformation $w$ of $V$, its \emph{fixed
    space} $\fix(w)$ is the set of vectors $\lambda \in V$ such that
  $w(\lambda) = \lambda$.  In other words, it is the kernel $\Ker(w -
  1)$.  Its \emph{move-set} $\mov(w)$ is the set of vectors $\mu$ in
  $V$ for which there exists some $\lambda$ in $V$ such that
  $w(\lambda) = \lambda + \mu$.  In other words, it is the image
  $\im(w - 1)$.
\end{defn}

\begin{remark}[Orthogonal transformations]\label{rem:orthogonal complement}
  Every element in a spherical Coxeter group, constructed from a root
  system as above, is an orthogonal transformation, and so the
  move-set and fixed space are orthogonal complements with respect to
  $V$.
\end{remark}

\subsection{Points, vectors, and affine Coxeter groups}\label{sec:pt-vec}
  
In the same way that doing linear algebra using fixed coordinate
systems can obscure an underlying geometric elegance, working in
affine Coxeter groups with a predetermined origin can have an
obfuscating effect.  One way to avoid making such a choice is to
distinguish between \emph{points} and \emph{vectors}, as in
\cite{SnapperTroyer-89} and \cite{BradyMcCammond-15}.

\begin{defn}[Points and vectors]\label{def:pt-vec}
  Let $V$ be a euclidean vector space.  A \emph{euclidean space} is a
  set $E$ with a uniquely transitive $V$-action, i.e., for every
  ordered pair of points $x,y \in E$ there exists a unique vector
  $\lambda \in V$ such that the $\lambda$ sends $x$ to $y$.  When this
  happens we write $x + \lambda = y$.  The preceding sentences
  illustrate two conventions that we adhere to throughout the paper:
  the elements of $E$ are \emph{points} and are denoted by Roman
  letters, such as $x$ and $y$, while the elements of $V$ are
  \emph{vectors} and are denoted by Greek letters, such as $\lambda$
  and $\mu$.
\end{defn}

The main difference between $E$ and $V$ is that $V$ has a well-defined
origin, but $E$ does not.  If we select a point $x \in E$ to serve as
the origin, then $V$ and $E$ can be identified by sending each vector
$\lambda$ to the point $x+\lambda$.  We use this identification in the
construction of the affine Coxeter groups.

\begin{defn}[Affine Coxeter groups]\label{def:affine}
  Let $E$ be a euclidean space, whose associated vector space $V$
  contains the crystallographic root system $\Phi$.  An affine Coxeter
  group $W$ can be constructed from $\Phi$, as follows.  Fix a point
  $x$ in $E$, to temporarily identify $V$ and $E$.  For each $\alpha
  \in \Phi$ and $j\in \Z$, let $H_{\alpha,j}$ denote the (affine)
  \emph{hyperplane} in $E$ of solutions to the equation $\langle
  v,\alpha \rangle = j$, where the brackets denote the standard inner
  product (treating $x$ as the origin).  The unique nontrivial
  isometry of $E$ that fixes $H_{\alpha,j}$ pointwise is a
  \emph{reflection} that we call $r_{\alpha,j}$.  The collection $R =
  \{r_{\alpha,j} \mid \alpha \in \Phi, j \in \Z\}$ generates the
  affine Coxeter group $W$ and $R$ is its set of reflections in the
  sense of Definition~\ref{def:lr}.  A standard minimal generating set
  $S$ can be obtained by restricting to those reflections that reflect
  across the facets of a certain polytope in $E$. The irreducible
  affine hyperplane arrangements of rank $2$ are shown in Figure
  \ref{fig:hyperplanes-infinite}.
\end{defn}

\begin{figure}[!h]
  \begin{tabular}{c c}
    (a) &
    \begin{tikzpicture}[cm={1,0,.5,.8660254,(0,0)}, >=stealth,baseline=2.6cm, scale=.8]
      \foreach \x in {1,...,6}{
        \draw (\x-1,.8)--(\x-1,6.2);
      }
      \draw (-1,2.8)--(-1,6.2);
      \draw (-2,4.8)--(-1-1,6.2);
      \draw (6,.8)--(6,4.2);
      \draw (7,.8)--(7,2.2);
      \foreach \x in {2,...,6}{
        \draw (1.2+\x,.8)--(-4.2+\x,6.2);
      }
      \draw (2.2,.8)--(-2.2,5.2);
      \draw (1.2,.8)--(-1.2,3.2);
      \draw (.2,.8)--(-.2,1.2);
      \draw (7.2,1.8)--(2.8,6.2);
      \draw (6.2,3.8)--(3.8,6.2);
      \foreach \x in {1,...,6}{
        \draw (.2-.5*\x,\x)--(8.2-.5*\x,\x);
      }
    \end{tikzpicture}
    \\
    \\
    (b) & 
    \begin{tikzpicture}[>=stealth,baseline=2.5cm,scale=.68]
      \foreach \x in {1,...,10}{
        \draw (\x,.8)--(\x,6.2);
      }
      \foreach \x in {1,...,6}{
        \draw (.8,\x)--(10.2,\x);
      }
      \foreach \x in {1,...,3}{
        \draw (.8,2*\x+.2)--(2*\x+.2,.8);
        \draw (.8,2*\x-.2)--(7.2-2*\x,6.2);
        \draw (2*\x+3.8,.8)--(10.2,7.2-2*\x);
      }
      \foreach \x in {1,...,2}{
        \draw (2*\x+.8,6.2)--(2*\x+6.2,.8);
        \draw (2*\x+4.8,6.2)--(10.2,2*\x+.8);
        \draw (2*\x+5.2,6.2)--(2*\x-.2,.8);
      }
    \end{tikzpicture}
    \\
    \\
    (c)
    &
    \begin{tikzpicture}[cm={1,0,.5,.8660254,(0,0)}, >=stealth,baseline=2.6cm,scale=.8]
      \foreach \x in {1,...,6}{
        \draw (\x-1,.8)--(\x-1,6.2);
      }
      \draw (-1,2.8)--(-1,6.2);
      \draw (-2,4.8)--(-1-1,6.2);
      \draw (6,.8)--(6,4.2);
      \draw (7,.8)--(7,2.2);
      \foreach \x in {2,...,6}{
        \draw (1.2+\x,.8)--(-4.2+\x,6.2);
      }
      \draw (2.2,.8)--(-2.2,5.2);
      \draw (1.2,.8)--(-1.2,3.2);
      \draw (.2,.8)--(-.2,1.2);
      \draw (7.2,1.8)--(2.8,6.2);
      \draw (6.2,3.8)--(3.8,6.2);
      \foreach \x in {1,...,6}{
        \draw (.2-.5*\x,\x)--(8.2-.5*\x,\x);
      }
      \draw (.8,.8)--(5.5,5.5);
      \draw (3.2,6.2)--(-.8,2.2);
      \draw (0.2,6.2)--(-1.8,4.2);
      \draw (3.8,.8)--(6.5,3.5);
      \draw (6.8,.8)--(7.5,1.5);
      \draw (-2.6,5.8)--(7.4,.8);
      \draw (-1.6,3.8)--(4.4,.8);
      \draw (-.6,1.8)--(1.4,.8);
      \draw (-.4,6.2)--(7,2.5);
      \draw (2.6,6.2)--(6,4.5);
      \draw (1.1,.8)--(-1.6,6.2);
      \draw (2.6,.8)--(-.1,6.2);
      \draw (4.1,.8)--(1.4,6.2);
      \draw (5.6,.8)--(2.9,6.2);
      \draw (7.1,.8)--(4.4,6.2);
    \end{tikzpicture}
  \end{tabular}
  \caption{ Affine hyperplanes for the root system of type (a) $A_2$,
    (b) $B_2$, and (c) $G_2$.}
  \label{fig:hyperplanes-infinite}
\end{figure}
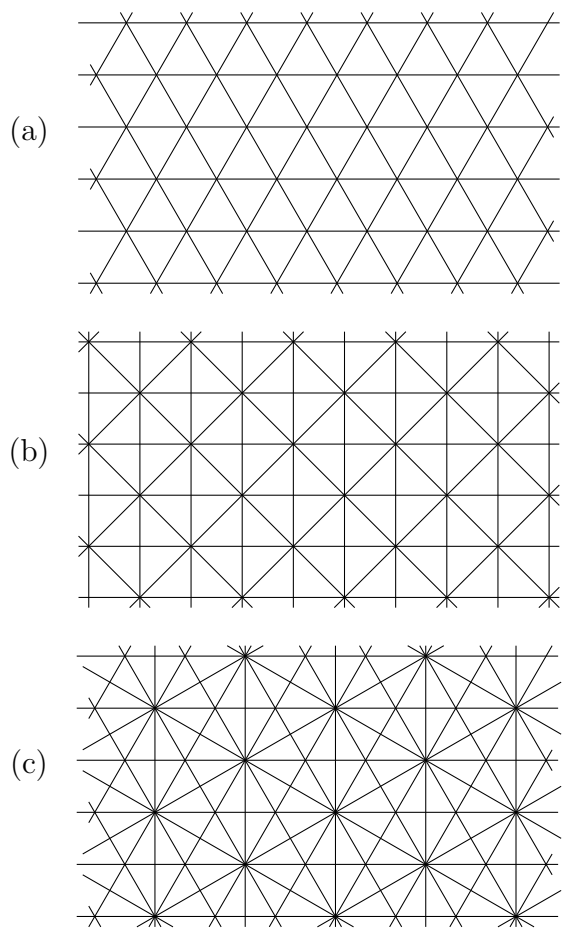

The affine Coxeter group $W$ associated to a finite crystallographic
root system $\Phi$ is closely related to the spherical Coxeter group
$W_0$.

\begin{defn}[Subgroups and quotients]\label{def:sub-quo}
  Let $W$ be an affine Coxeter group constructed as in
  Definition~\ref{def:affine}.  The map sending $r_{\alpha,j}$ in $W$
  to $r_\alpha$ in $W_0$ extends to a group homomorphism $p\colon W
  \onto W_0$.  The kernel of $p$ is a normal abelian subgroup $T$,
  isomorphic to $\Z^n$, whose elements are called \emph{translations},
  and $W_0 \cong W/T$.  When $x$ is the point used in the construction
  of $W$, the map $i\colon W_0 \into W$ sending $r_\alpha$ to
  $r_{\alpha, 0}$ is a section of the projection $p$, identifying
  $W_0$ with the subgroup of all elements of $W$ that fix $x$. (Note
  also the composition $i\circ p$ sends reflection $r_{\alpha,j}$ to
  $r_{\alpha}$ to $r_{\alpha,0}$.) Thus $W$ may be identified as a
  semidirect product $W \cong T \rtimes W_0$.
\end{defn}
  
Of course, the identification of $W_0$ with a subgroup of $W$ is not
unique: conjugation by elements of $T$ gives an infinite family of
such subgroups.

\begin{figure}[h]
  \begin{tikzpicture}[>=stealth]
    \draw (0,0) node{
      \begin{tikzpicture}[baseline=2.5cm,scale=.8]
        \draw[draw=none,fill=white!80!blue] (3,4)--(4,4)--(4,3)--(3,4);
        \foreach \x in {1,...,5}{
          \draw (\x,1.8)--(\x,6.2);
          \draw (.8,\x+1)--(5.2,\x+1);
        }
        \foreach \x in {1,...,3}{
          \draw (.8,2*\x+.2)--(2*\x-.8,1.8);
          \draw (.8,2*\x-.2)--(7.2-2*\x,6.2);
        }
        \foreach \x in {1,...,2}{
          \draw (2*\x+.8,6.2)--(5.2,2*\x+1.8);
          \draw (5.2,6.2-2*\x)--(2*\x+.8,1.8);
        }
        \draw[very thick] (.5,6.5)--(5.5,1.5);
        \draw[very thick] (0,4)--(6,4);
        \draw (5,4) node[circle,inner sep=2, fill=black] {};
        \draw (5,6) node[circle,inner sep=2, fill=black] {};
        \draw (5,2) node[circle,inner sep=2, fill=black] {};
        \draw (3,6) node[circle,inner sep=2, fill=black] {};
        \draw (3,4) node[circle,inner sep=2, fill=black] {};
        \draw (3,2) node[circle,inner sep=2, fill=black] {};
        \draw (1,2) node[circle,inner sep=2, fill=black] {};
        \draw (1,4) node[circle,inner sep=2, fill=black] {};
        \draw (1,6) node[circle,inner sep=2, fill=black] {};
      \end{tikzpicture}
    };
    \draw (0,0) circle (4cm);
    \foreach \x in {1,...,8}{
      \draw (45*\x:4cm) node[circle,inner sep=2, fill=black] {}; 
      \draw[dashed] (45*\x:3.25cm)--(45*\x:4.5cm);
    }
    \draw[dashed,->] (.8,-.25) node[inner sep=1,left] {$A$} --
    (-20:4.5cm) node[below right,inner sep=1] {$p(A)=C$}; 
  \end{tikzpicture}
  \caption{Elements of $W$ can be put in bijection with alcoves in
    tessellated plane and the orbit of $x$ under translations in $W$
    is illustrated with the fat vertices of these alcoves.  Elements
    of $W_0$ are in bijection with chambers on the boundary sphere.
    The projection map $p$ maps an alcove $A$ to the chamber $C$ at
    infinity that points ``in the same direction''. For details see
    Remark~\ref{rem:alcoves}. }\label{fig:projection}
\end{figure}
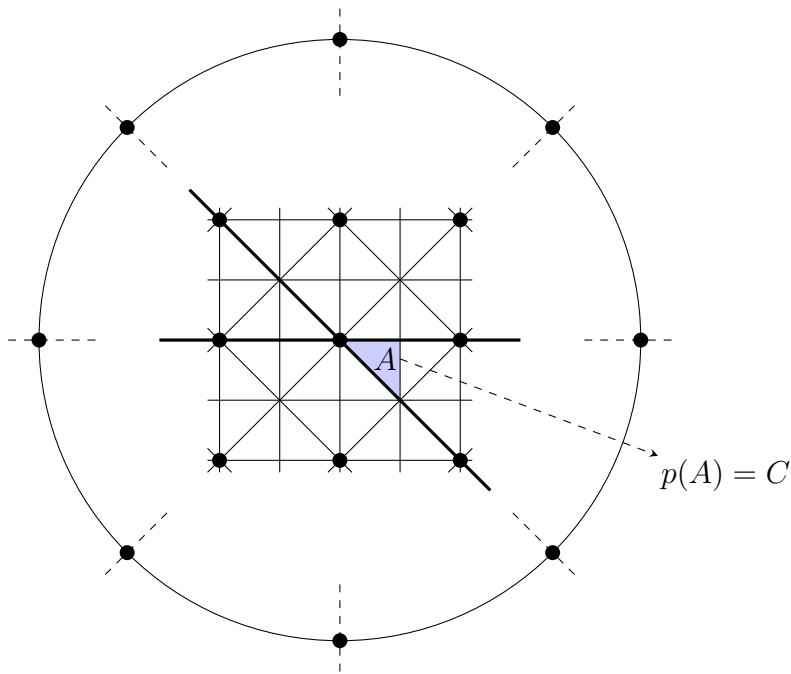

\begin{remark}[Geometric interpretation of $W$, $W_0$ and $p$]\label{rem:alcoves}
  Given $W$, an \emph{alcove} is (the closure of) a maximal
  connected component in the complement of the reflection hyperplanes
  $H_{\alpha,j}$ in the euclidean space $E$. These are the small
  triangles in Figures \ref{fig:hyperplanes-infinite} and
  \ref{fig:projection}. Each of these triangles can play the role of a
  fundamental domain of the action of $W$ on $E$.
  The choice of a point $x$ as origin and a fundamental alcove $c_x$
  having $x$ as a vertex determines a bijection between elements of
  $W$ and alcoves. Here the fundamental alcove corresponds to the
  identity and every element $w\in W$ corresponds to the image of
  $c_x$ under $w$.

  The boundary of $E$ is a sphere $S$ whose points are the parallelism
  classes of geodesic rays. Its simplicial structure is induced by the
  tessellation of the plane. A parallelism class of hyperplanes in $E$
  corresponds to a reflection hyperplane, i.e., an equator, in the
  sphere $S$.  We call the maximal connected components in the
  complement of the hyperplanes in $S$ \emph{chambers}.  Each chamber
  is also the parallelism class of simplicial cones in $E$. Namely, for
  any one of the black vertices in Figure \ref{fig:projection}, 
  say $y$, the complement of all the
  hyperplanes going through $y$ decomposes into simplicial cones. We
  call them \emph{Weyl cones} based at $y$. The parallelism classes of
  Weyl cones are in bijection with the chambers in $S$.  Moreover, the
  chambers in $S$ are in a natural bijection with elements of $W_0$, as
  follows: Each element $w\in W$ is an affine motion of $E$ that
  obviously maps parallel rays to parallel rays. Hence it induces an
  isometry of (the tessellation of) $S$ where translations on $E$
  induce the identity on $S$. If we want the bijection between $W_0$
  and the chambers in $S$ to be compatible with these induced
  isometries of $W$ on $S$ we need to map the identity in $W_0$ to the
  parallelism class of the unique Weyl cone based at $x$ that contains
  the alcove $c_x$.

  From this perspective, the projection map $p: W \onto W_0$ can be
  understood geometrically as the map that sends an element $w\in W$
  to the induced isometry on $S$. In geometric terms this means that
  an alcove $A$ with vertex $y$ is mapped to the chamber $C$ at
  infinity that is the parallelism class of the cone based at $y$ that
  contains $A$. This is indicated by the dotted arrow in
  Figure~\ref{fig:projection}. One can think of this as ``walking to
  infinity in the direction of $A$ at $y$''.  More about the interplay
  between faces of alcoves and faces of chambers is studied in work of
  Marcelo Aguiar and the third author \cite{AguiarPetersen-15}.
\end{remark}

The notions of fixed space and move-set extend easily to affine
Coxeter groups.

\begin{defn}[Move-sets]\label{def:move-sets}
  The \emph{motion} of a point $x \in E$ under a euclidean isometry
  $w$ is the vector $\lambda \in V$ such that $w(x) = x + \lambda$.
  The \emph{move-set} of $w$ is the collection of all motions of the
  points in $E$; by \cite[Proposition~3.2]{BradyMcCammond-15}, it is
  an affine subspace of $V$.
  In symbols, $\mov(w) = \{ \lambda \mid w(x) = x+\lambda
  \textrm{ for some } x \in E \} \subset V$.
\end{defn}

\begin{defn}[Fixed space]\label{def:fixed-space}
  The \emph{fixed space} $\fix(w)$ of an isometry $w$ is the subset of
  points $x \in E$ such that $w(x) = x$.  Equivalently, $\fix(w)$
  consists of all points whose motion under $w$ is the vector $0$.
  When $\fix(w)$ is nonempty, it is an (affine) subspace of $E$.
\end{defn}

The complementarity between $\mov(w)$ and $\fix(w)$ mentioned in
Remark~\ref{rem:orthogonal complement} does not hold for isometries of
euclidean space.  It can, however, be recovered if the fixed space is
replaced with the min-set of points that are moved a minimal distance
under $w$.  Then the space of directions for the move-set and the
space of directions for the min-set give an orthogonal decomposition
of $V$ -- see \cite[Lemma 3.6]{BradyMcCammond-15}.

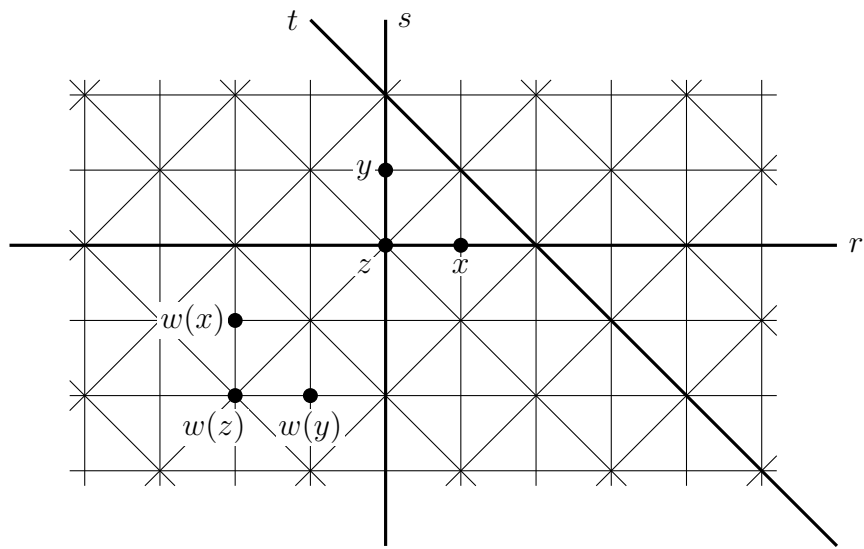
\begin{figure}
  \begin{tikzpicture}[>=stealth,baseline=2.5cm,scale=1]
    \foreach \x in {1,...,10}{
      \draw (\x,.8)--(\x,6.2);
    }
    \foreach \x in {1,...,6}{
      \draw (.8,\x)--(10.2,\x);
    }
    \foreach \x in {1,...,3}{
      \draw (.8,2*\x+.2)--(2*\x+.2,.8);
      \draw (.8,2*\x-.2)--(7.2-2*\x,6.2);
      \draw (2*\x+3.8,.8)--(10.2,7.2-2*\x);
    }
    \foreach \x in {1,...,2}{
      \draw (2*\x+.8,6.2)--(2*\x+6.2,.8);
      \draw (2*\x+4.8,6.2)--(10.2,2*\x+.8);
      \draw (2*\x+5.2,6.2)--(2*\x-.2,.8);
    }
    \draw[very thick] (5,0)--(5,7) node[right] {$s$};
    \draw[very thick] (0,4)--(11,4) node[right] {$r$};
    \draw[very thick] (4,7) node[left] {$t$}--(11,0);
    \draw (5,4) node[circle,inner sep=2, fill=black] {} node[xshift=-8pt,yshift=-8pt,inner sep=1,fill=white] {$z$};
    \draw (5,5) node[circle,inner sep=2, fill=black] {} node[xshift=-8pt,inner sep=1,fill=white] {$y$};
    \draw (6,4) node[circle,inner sep=2, fill=black] {} node[yshift=-8pt,inner sep=1,fill=white] {$x$};
    \draw (3,2) node[circle,inner sep=2, fill=black] {} node[xshift=-8pt,yshift=-12pt,inner sep=1,fill=white] {$w(z)$};
    \draw (4,2) node[circle,inner sep=2, fill=black] {} node[yshift=-12pt,inner sep=1,fill=white] {$w(y)$};
    \draw (3,3) node[circle,inner sep=2, fill=black] {} node[xshift=-16pt,inner sep=1,fill=white] {$w(x)$};
  \end{tikzpicture}
  \caption{The action of $w = rst$ in the group of affine type $B_2$
    on the points $x, y, z$ in the euclidean plane, as in
    Example~\ref{ex:B_2}.}\label{fig:B2example}
\end{figure}

\begin{example}[The move-set of an element in affine $B_2$]\label{ex:B_2}
  To make the notion of a move-set concrete, we describe the move-set
  of an element in affine $B_2$. Suppose $w = rst$, where $r$, $s$,
  and $t$ are reflections across the lines pictured in Figure
  \ref{fig:B2example}, and suppose $x, y,$ and $z$ are the points
  labeled. We have also labeled the locations of $w(x)$, $w(y)$ and
  $w(z)$.
  
  The standard basis vectors $\varepsilon_1$ and $\varepsilon_2$ in
  $\R^2$ are the vectors that send $z$ to $x$ and $z$ to $y$,
  respectively. By abuse of notation, we will write $\varepsilon_1 =
  x-z$ and $\varepsilon_2 = y-z$. Having made this identification, we
  can express any point $u$ in $E\cong \R^2$ as 
  \[ 
  u = z + a(x-z) + b(y-z) 
  \] 
  for some $(a,b)\in \R^2$. Using this coordinate system, we have
  \begin{align*}
    w(x) &= z-2(x-z)-(y-z) = x-3(x-z)-(y-z),\\
    w(y) &=z-(x-z)-2(y-z) =y -(x-z)-3(y-z), \quad \mbox{and}\\
    w(z) &=z-2(x-z)-2(y-z).
  \end{align*}
  We see that the motion of $x$ is the vector $(-3,-1)$, the motion of
  $y$ is $(-1,-3)$, and the motion of $z$ is $(-2,-2)$.

  In general, we use linearity to compute
  \begin{align*}
    w(u) &= w(z) +a w(x-z) + bw(y-z),\\
    &=w(z) +a(w(x)-w(z)) + b(w(y)-w(z)),\\
    &=z-2(x-z)-2(y-z) + a(y-z) + b(x-z),\\
    &=u+(b-a-2)(x-z) + (a-b-2)(y-z).
  \end{align*}
  Thus, a generic motion is $\lambda = (-2,-2) + (b-a)\cdot (1,-1)$,
  and the move-set of this element $w$ is an affine line:
  \[
  \mov(w) = (-2,-2)+\R(1,-1).
  \]
\end{example}

Move-sets for general euclidean plane isometries are given in
Example~\ref{ex:euc-plane}.

\subsection{Elliptics and translations}\label{sec:elliptic}

In this section, we record some basic facts about special kinds of
elements in an affine Coxeter group.

\begin{defn}[Roots and reflections]\label{def:roots}
  For a reflection $r$ whose fixed space is the hyperplane $H$,
  the motion of any point under $r$ is in a direction orthogonal to $H$,
  and $\mov(r)$ is the line through the origin in $V$ in this direction.
  A vector $\alpha \in V$ is called a
  \emph{root} of $r$ if $\mov(r) = \spn(\alpha) = \R \alpha$.
  For affine Coxeter groups, the fixed spaces of the reflections come
  in finitely many parallel families.  Within each family, the fixed
  hyperplanes are equally spaced and the length of their common root
  $\alpha$ can be chosen to encode additional information such as the
  distance between adjacent parallel fixed hyperplanes.  Once the
  roots are normalized in this way, the result is 
  the usual crystallographic root system $\Phi$ inside $V$. The
  $\Z$-span of $\Phi$ in $V$ is called the \emph{root lattice}.
  It is an abelian group (under vector addition) 
  isomorphic to $\Z^n$ where $n = \dim(V)$.
\end{defn}

\begin{defn}[Elliptic elements and elliptic part]\label{def:ell-part}
  An element $w$ in an affine Coxeter group $W$ is called
  \emph{elliptic} if its fixed space is non-empty.  Equivalently,
  these are exactly the elements of $W$ of finite order.
  Given an arbitrary element $w \in W$, its \emph{elliptic part}
  $w_e = p(w)$ is its image under the projection $p: W \onto W_0$.
  (In particular, the elliptic part is an element of $W_0$, acting naturally
  on $V$ rather than on $E$.)
\end{defn}

One can characterize elliptic elements in a variety of ways.

\begin{lem}[Elliptic elements]\label{lem:ell}
  For a euclidean isometry $w$, the following are equivalent: (1) $w$
  is elliptic, (2) $\mov(w) \subset V$ is a linear subspace, and (3)
  $\mov(w)$ contains the origin.
\end{lem}

\begin{proof}
  \cite[Proposition~3.2 and Definition~3.3]{BradyMcCammond-15}.
\end{proof}

\begin{defn}[Coroots]\label{def:coroots}
  Let $\Phi \subset V$ be a crystallographic root system. For each root
  $\alpha \in \Phi$, its \emph{coroot} is the vector $\alpha^\vee =
  \frac{2}{\langle \alpha, \alpha \rangle} \alpha$.  The collection of
  these coroots is another crystallographic root system $\Phi^\vee =
  \{ \alpha ^\vee \mid \alpha \in \Phi\}$.  The $\Z$-span $L(\Phi^\vee)$ of
  $\Phi^\vee$ in $V$ is the \emph{coroot lattice}; as an abelian 
  group, it is also isomorphic to $\Z^n$ where $n = \dim(V)$.
\end{defn}

\begin{defn}[Translations]\label{def:trans}
  For every vector $\lambda \in V$ there is a euclidean isometry
  $t_\lambda$ of $E$ called a \emph{translation} that sends each point $x
  \in E$ to $x + \lambda$.  Let $W$ be an affine Coxeter group acting
  on $E$ with root system $\Phi \subset V$.  An element of $W$ is a
  translation in this sense if and only if it is in the kernel $T$ of the
  projection $p\colon W \onto W_0$.  Moreover, the set of vectors in
  $V$ that define the translations in $T$ is identical to the set of
  vectors in the coroot lattice $L(\Phi^\vee)$.
\end{defn}

The next result follows immediately from the definitions.

\begin{prop}[Translated move-sets]
\label{prop:translated move-set}
  For any euclidean isometry $w$ and any translation $t_\lambda$, 
  $\mov(t_\lambda w) = \lambda + \mov(w)$.
\end{prop}

\begin{defn}[Translation-elliptic factorizations]\label{def:trans-ell}
  Let $W$ be an affine Coxeter group.  There are many ways to write an
  element $w \in W$ as a product of a translation $t_\lambda \in W$
  and an elliptic $u \in W$.  We call any such factorization $w =
  t_\lambda u$ a \emph{translation-elliptic factorization} of $w$.
  In any such factorization we call
  $t_\lambda$ the \emph{translation part} and $u$ the \emph{elliptic
    part}.
\end{defn}

\begin{defn}[Normal forms]\label{def:semidirect}
  Let $W$ be an affine Coxeter group acting cocompactly on a euclidean
  space $E$.  An identification of $W$ as a semidirect product $T
  \rtimes W_0$ corresponds to a choice of an inclusion map $i\colon
  W_0 \into W$ that is a section of the projection map $p\colon W \onto
  W_0$.  The unique point $x \in E$ fixed by the subgroup $i(W_0)$
  serves as our origin and every element has a unique factorization $w
  = t_\lambda u$ where $t_\lambda$ is a translation in $T$ and $u$ is
  an elliptic element in $i(W_0)$. In
  particular, $u = i(w_e)$ is the image under $i$ of the elliptic part
  of $w$ (Definition~\ref{def:ell-part}).  For a fixed choice of a
  section $i$, we call this unique factorization $w = t_\lambda u$ the
  \emph{normal form} of $w$ under this identification.
\end{defn}

If $w = t_\lambda u$ is a translation-elliptic factorization then $u_e
= w_e \in W_0$ and $t_\lambda$ is in the kernel of $p$.  Some
translation-elliptic factorizations come from an identification of $W$
and $T \rtimes W_0$ as in Definition~\ref{def:semidirect}, but not all
of them do: see Example~\ref{ex:insuff}.

\begin{prop}[Recognizing elliptics]\label{prop:ell-recog}
  Let $W$ be an affine Coxeter group and let $w = t_\lambda u$ be a
  translation-elliptic factorization of an element $w \in W$.  The
  following are equivalent: (1) $w$ is elliptic, (2) $\lambda \in
  \mov(u)$, (3) $\mov(w) = \mov(u)$, and (4) $\mov(w) = \mov(w_e)$.
\end{prop}

\begin{proof}
  By Proposition~\ref{prop:translated move-set}, $\mov(w) = \lambda +
  \mov(u)$.  Since $u$ is elliptic, $\mov(u)$ is a linear subspace of
  $V$ (Lemma~\ref{lem:ell}).  Thus, $\mov(w)$ is a linear subspace of
  $V$ if and only if $\lambda \in \mov(u)$, so (1) and (2) are
  equivalent.  Moreover, $\lambda \in \mov(u)$ if and only if $\mov(u)
  = \lambda + \mov(u) = \mov(w)$, so (2) and (3) are equivalent.
  Finally, for any elliptic element $w$, choose a point $x$ fixed by
  $w$; then $w(x + \lambda) = x + w_e(\lambda)$ for all vectors
  $\lambda$.  Thus $w(x + \lambda) = (x + \lambda) + (w_e(\lambda) -
  \lambda)$, so $\mu$ belongs to $\mov(w)$ if and only if $\mu$ is of
  the form $w_e(\lambda) - \lambda$, i.e., if $\mu$ belongs to
  $\mov(w_e)$.  Thus (1) and (4) are equivalent.
\end{proof}

When isometries are composed, the new move-set is contained in the
vector sum of the affine subspaces that are the individual move-sets.

\begin{prop}[Move-set addition]\label{prop:move-add}
  If $w_1$ and $w_2$ are elements of an affine Coxeter group $W$, then
  $\mov(w_1 \cdot w_2) \subset \mov(w_1) + \mov(w_2)$.  Moreover, if
  $r$ is a reflection and its root $\alpha$ is not in $\mov(w_e)$, 
  then $\mov(wr) = \mov(rw) = \mov(r) + \mov(w)$.
\end{prop}

\begin{proof}
  The motion vector of a point $x$ under the product $w_1w_2$ is the
  motion vector of $x$ under $w_2$ plus the motion vector of $w_2(x)$
  under $w_1$.  The second assertion is part of
  \cite[Proposition~6.2]{BradyMcCammond-15}.
\end{proof}

This quickly leads to Carter's elegant geometric characterization of
the reflection length in spherical Coxeter groups \cite{Carter72} and
its extension to reflection length for elliptic elements in affine
Coxeter groups.

\begin{lemma}[Factoring elliptics]\label{lem:red-ell}
  Let $w =r_1 r_2 \cdots r_k$ be a product of reflections in an affine Coxeter 
  group, where reflection $r_i$ is through an affine hyperplane $H_i$ 
  orthogonal to the root $\alpha_i$.  If $w$ is elliptic and $\lr(w)=k$ then the
  roots $\alpha_i$ are linearly independent.  Conversely, if the roots
  $\alpha_i$ are linearly independent then $w$ is elliptic, $\lr(w)=k$, 
  $\fix(w) = H_1 \cap \cdots \cap H_k$, and $\mov(w) =
  \spn(\{\alpha_1,\alpha_2,\ldots,\alpha_k\})$.
\end{lemma}

\begin{proof}
  In \cite[Lemmas~3.6 and~6.4]{BradyMcCammond-15}, these facts are
  proved in the full isometry group of the euclidean space $E$
  generated by all possible reflections, but the statements and their
  proofs easily restrict to the case where every reflection is in the
  set $R \subset W$ and reflection length is computed with respect to
  this collection of reflections.
\end{proof}

\begin{rem}[Maximal elliptics]\label{rem:max-ell}
  Let $W$ be an affine Coxeter group acting cocompactly on an
  $n$-dimensional euclidean space $E$.  When $u$ is an elliptic
  element of reflection length~$n = \dim(E)$ (such as a Coxeter
  element of a maximal parabolic subgroup of $W$), its move-set
  $\mov(u)$ is all of $V$ (Lemma~\ref{lem:red-ell}).  By the
  equivalence of parts (1) and (2) in
  Proposition~\ref{prop:ell-recog}, the elements $t_\lambda u$ are
  elliptic for all choices of translation $t_\lambda\in T$ and, in
  particular, $\lr(t_\lambda u) = \lr(u) = n$ for every $t_\lambda \in
  T$.
\end{rem}

\subsection{Dimension of an element}\label{sec:dim}

This section shows how to assign a dimension to each element in a
spherical or affine Coxeter group.  It is based on the relationship
between move-sets and root spaces.

\begin{defn}[Root spaces]\label{def:move-set-arr}
  Let $V$ be a euclidean vector space with root system $\Phi$.  A
  subset $U \subset V$ is called a \emph{root space} if it is the span
  of the roots it contains.  In symbols, $U$ is a root space when $U =
  \spn(U \cap \Phi)$.  Equivalently, $U$ is a root space when $U$ is a
  linear subspace of $V$ that is spanned by a collection of roots, or
  when $U$ has a basis consisting of roots.  Since $\Phi$ is a finite
  set, there are only finitely many root spaces; the collection of all
  root spaces in $V$ is called the \emph{root space arrangement}
  $\arr(\Phi) = \{ U \subset V \mid U = \spn(U \cap \Phi)\}$.
\end{defn}

\begin{defn}[Root dimension]\label{def:root-dim}
  For any subset $A \subset V$, we define its \emph{root dimension}
  $\dim_\Phi(A)$ to be the minimal dimension of a root space in
  $\arr(\Phi)$ that contains $A$.  Since $V$ itself is a root space,
  $\dim_\Phi(A)$ is defined for every subset $A$ in $V$.
\end{defn}

\begin{defn}[Dimension of an element]\label{def:dim-elt}
  When $w$ is an element of a spherical or affine Coxeter group, its
  move-set is contained in a euclidean vector space $V$ that also
  contains the corresponding root system $\Phi$.  The \emph{dimension}
  $\dim(w)$ of such an element is defined to be the root dimension of
  its move-set.  In symbols, $\dim(w) = \dim_\Phi(\mov(w))$.  Let $W$
  be an affine Coxeter group acting on a euclidean space $E$ and let
  $p\colon W \onto W_0$ be its projection map.  For each element $w
  \in W$ we can compute the dimension of $w$ and the dimension of its
  elliptic part $w_e = p(w) \in W_0$. We call $e = e(w) = \dim(w_e)$
  the \emph{elliptic dimension} of $w$.  Instead of focusing on the
  dimension of $w$ itself, we focus on the number $d = d(w) = \dim(w)
  - \dim(w_e)$, which we call the \emph{differential dimension} of
  $w$.  Note that $\dim(w) = d+e$.
\end{defn}

The statistics $d(w)$ and $e(w)$ carry geometric meaning.

\begin{prop}[Statistics and geometry]\label{prop:stat-geo}
  Let $W$ be an affine Coxeter group.  An element $w\in W$ is a
  translation if and only if $e(w) = 0$, and $w$ is elliptic if and
  only if $d(w) = 0$.
\end{prop}

\begin{proof}
  If $w$ is a translation, then it is in the kernel of the projection $p$,
  $w_e$ is the identity, and $e(w) = 0$.  Conversely, if $e(w) = 0$
  then $\mov(w_e) = \{0\}$, $w_e$ is the identity, $w$ is in the
  kernel of $p$ and thus is a translation.  
  
  If $w$ is elliptic, then by Proposition~\ref{prop:ell-recog}
  $\mov(w) = \mov(w_e)$, so $\dim(w) = \dim(w_e)$ and $d(w) = 0$.
  Conversely, if $d(w) =0$ then $\dim(w) = \dim(w_e)$.  By
  Lemma~\ref{lem:red-ell}, $\mov(w_e)$ is itself a root subspace, so
  we must have $\mov(w) = \mov(w_e)$.  Then $w$ is elliptic by
  Proposition~\ref{prop:ell-recog}.
\end{proof}

Thus one way to interpret these numbers is that, roughly speaking,
$d(w)$ measures how far $w$ is from being an elliptic element and
$e(w)$ measures how far $w$ is from being a translation.  We record one
more elementary fact for later use.

\begin{lem}[Separation]\label{lem:sep}
  Let $M$ and $U$ be linear subspaces of a vector space $V$ and let
  $\lambda$ be a vector.  The space $M$ contains $\lambda + U$ if and
  only if $M$ contains both $U$ and $\lambda$.
\end{lem}

\begin{proof}
  If $M$ contains $\lambda + U$ then it contains $\lambda$ (since $U$
  contains the origin) and it contains $-\lambda$ (since $M$ is closed
  under negation).  Thus it contains $(\lambda + U) + (-\lambda) = U$.
  The other direction is immediate.
\end{proof}

The next two propositions record the basic relationships between
these dimensions and reflection length.

\begin{prop}[Inequalities]\label{prop:inequal}
  Let $W$ be an affine Coxeter group.  For every element $w \in W$,
  $\lr(w) \geq \dim(w) \geq \dim(w_e) = \dim(U)$ where $U = \mov(w_e)$
  is the move-set of the elliptic part of $w$.
\end{prop}

\begin{proof}
  If $w$ can be written as a product of $k$ reflections, then
  $\mov(w)$ is contained in the (linear) span of their roots
  (Proposition~\ref{prop:move-add}).  Thus $k \geq \dim(w)$, and
  choosing a minimum-length reflection factorization for $w$ shows
  that $\lr(w) \geq \dim(w)$.  Next, if $w = t_\lambda u$ is a
  translation-elliptic factorization of $w$, then $\mov(w) = \lambda +
  U$, where $U = \mov(u) = \mov(w_e)$
  (Proposition~\ref{prop:translated move-set}).  By
  Lemma~\ref{lem:sep}, any root space $M$ that contains $\mov(w) =
  \lambda + U$ also contains $\mov(w_e) = U$, which proves $\dim(w)
  \geq \dim(w_e)$.  Finally, since $U$ is itself a root subspace by
  Lemma~\ref{lem:red-ell}, we have $\dim(w_e) = \dim(U)$.
\end{proof}

\begin{prop}[Elliptic equalities]\label{prop:equal}
  Let $W$ be an affine Coxeter group.  When $w$ is elliptic, $\lr(w) =
  \dim(w) = \dim(w_e) = \dim(U)$ where $U = \mov(w) = \mov(w_e)$.
\end{prop}

\begin{proof}
  By Proposition~\ref{prop:ell-recog}, $\mov(w) = \mov(w_e) = U$.  By 
  Lemma~\ref{lem:red-ell}, $U$ is
  a root space and $\lr(w) = \dim(U)$.  Proposition~\ref{prop:inequal}
  completes the proof.
\end{proof}

\begin{table}
  \begin{tabular}{|c|c|cc|c|}
    \hline
    $w$ & $\mov(w)$ & $d$ & $e$ & $\lr$\\
    \hline
    identity & the origin & 0 & 0 & 0 \\
    reflection & a root line & 0 & 1 & 1 \\
    rotation & the plane & 0 & 2 & 2 \\
    \hline
    translation & an affine point & 1 or 2 & 0 & 2 or 4\\
    \hline
    glide reflection & an affine line & 1 & 1 & 3 \\
    \hline
  \end{tabular}
  \vspace*{1em}
  \caption{Basic invariants for the $5$ types of elements in an affine
    Coxeter group acting on the euclidean plane.}\label{tab:euc-plane}
\end{table}

\begin{example}[Euclidean plane]\label{ex:euc-plane}
  The affine Coxeter groups that act on the euclidean plane have five
  different types of move-sets (see Table~\ref{tab:euc-plane}).  
  Among the elliptic elements, the move-set of the identity is the
  point at the origin, the move-set of a reflection $r$ with root
  $\alpha \in \Phi$ is the root line $\R\alpha$, and the move-set of
  any non-trivial rotation is all of $V=\R^2$.  For these elements, 
  $d(w)= 0$ and $\lr(w) = \dim(w) = \dim(w_e) = e(w)$
  where this common value is $0$, $1$, or $2$, respectively.  The
  move-set of a non-trivial translation $t_\lambda$ is the
  single nonzero vector $\{\lambda\}$.  Its elliptic
  dimension is $0$ and its differential dimension is either $1$ 
  (when $\lambda$ is contained in a root line $\R\alpha$) or $2$.  By
  \cite[Proposition~4.3]{McPe-11}, $\lr(t_\lambda)$ is twice its
  dimension. Finally, when $w$ is a glide reflection, $\mov(w)$ is a
  line not through the origin, so $e = \dim(w_e) = 1$, $\dim(w) = 2$, 
  $d=2-1=1$, and $\lr(w) = 3$.
\end{example}

We finish this section with a pair of remarks about computing $e$ and
$d$ in general.  Let $W$ be an affine Coxeter group with a fixed
identification of $W$ with $T \rtimes W_0$ and let $w \in W$ be an
element that is given in its semidirect product normal form $w =
t_\lambda u$ for some vector $\lambda$ in the coroot lattice
$L(\Phi^\vee)$ and some elliptic element $u$. Computing the elliptic
dimension $e(w)$ is straightforward.

\begin{rem}[Computing elliptic dimension]\label{rem:computing-e}
  To compute the elliptic dimesion $e(w)$ it is sufficient to simply
  compute the dimension of the move-set of the elliptic part $u$ of
  its normal form.  Indeed, by Definition~\ref{def:dim-elt},
  $e(w) = \dim(w_e)$, but since $\mov(w_e)$ is
  itself a root subspace by Lemma~\ref{lem:red-ell},
  $\dim_\Phi(\mov(w_e)) = \dim(\mov(w_e))$.  Finally, since $p(u) =
  p(w) = w_e$, $\mov(w_e) = \mov(u)$ by Proposition~\ref{prop:equal}
  and so $\dim(\mov(w_e)) = \dim(\mov(u))$.
\end{rem}

Computing the differential dimension $d(w)$ is more complicated but it
can be reduced to computing the dimension of a point in a simpler
arrangement of subspaces in a lower dimensional space.  How much lower
depends on the elliptic dimension $e(w)$.

\begin{rem}[Computing differential dimension]\label{rem:computing-d}
  By Definition~\ref{def:dim-elt}, to compute the differential 
  dimension $d(w)$, we need to find the minimal dimension of a root 
  subspace containing $\mov(w)$ and then subtract $e(w)$ from this
  value.  Since $w = t_\lambda u$, $\mov(w) = \lambda + U$ where $U =
  \mov(u) = \mov(w_e)$.  Moreover, by Lemma~\ref{lem:sep} we only need
  to consider root spaces that contain $\lambda$ and $U$ or,
  equivalently, root spaces that contain $\lambda + U$ and $U$.  Let
  $q\colon V \onto V/U$ be the natural quotient linear transformation
  whose kernel is $U$.  Under the map $q$, the coset $\lambda + U$ is
  sent to a point in $V/U$ that we call $\lambda/U$ and the subspaces in
  $\arr(\Phi)$ containing $U$ are sent to a collection of subspaces in
  $V/U$ that we call $\arr(\Phi/U)$.  Let $\dim_{\Phi/U}(\lambda/U)$
  be the minimal dimension of a subspace in $\arr(\Phi/U)$ that
  contains the point $\lambda/U$.  Since the dimensions involved have
  all been diminished by $e(w) = \dim(U)$, we have that
  $\dim_{\Phi/U}(\lambda/U) = d(w)$ is the differential dimension of
  $w$.
\end{rem}

\section{Proofs of main theorems}\label{sec:mains}

In the following subsections we prove our main results.

\subsection{Proof of Theorem~\ref{main:dim}}\label{subsec:A}

In this subsection, we prove Theorem~\ref{main:dim} by showing the
claimed value is both a lower bound and an upper bound for the
reflection length of $w$.  Let $W$ be an affine Coxeter group with
root system $\Phi$ and projection map $p\colon W \onto W_0$.  For each
element $w \in W$ we write $w_e = p(w)$ for the elliptic part of $w$,
$e = e(w) = \dim(w_e)$ for its elliptic dimension, and $d = d(w) =
\dim(w) - \dim(w_e)$ for its differential dimension.

\begin{prop}[Lower bound]\label{prop:thmA-lb}
  For every $w \in W$, $\lr(w) \geq 2d + e$.
\end{prop}

\begin{proof}
  Let $w = r_1 r_2 r_3 \cdots r_k$ be a reflection factorization of
  $w$. For each reflection $r_i$, let $\alpha_i$ be one of its
  roots.  Let $M$ be the span of this set of vectors $\alpha_i$ and
  let $m = \dim(M)$.  By Proposition~\ref{prop:move-add}, 
  $\mov(w) \subset M$ and thus $d + e = \dim(w) \leq m$.

  Next, pick a basis for $M$ from among the $\alpha_i$.  By
  Lemma~\ref{lem:hurwitz}, we may assume that the reflections
  corresponding to our chosen basis are the reflections $r_1, \ldots,
  r_m$.  Define $u = r_1 r_2\cdots r_m$ 
  and $v = r_{m+1} \cdots r_{k-1}r_k$, so that $w = uv$ and $u =
  wv^{-1}$.

  By Lemma~\ref{lem:red-ell}, $u$ is elliptic and $\mov(u) =\mov(u_e) =
  M$, and so $\dim(u) = \dim(u_e) = m$ (Proposition~\ref{prop:equal}).
  By Proposition~\ref{prop:inequal}, $ k-m \geq \dim(v) \geq \dim(v_e)
  = \dim(v_e^{-1})$ because $v$ is a product of $k-m$ reflections.
  Projecting $u = wv^{-1}$ gives $u_e = w_e v_e^{-1}$ and thus
  $\mov(u_e) \subset \mov(w_e) + \mov(v_e^{-1})$ by
  Proposition~\ref{prop:move-add}.  Taking dimensions shows that $m
  \leq e + (k - m)$.  Thus, $k \geq 2m - e \geq 2(d+e)-e = 2d + e$, as
  claimed.
\end{proof}

\begin{prop}[Upper bound]\label{prop:thmA-ub}
  For every $w\in W$, $\lr(w) \leq 2d+e$.
\end{prop}

\begin{proof}
  Let $w = t_\lambda u$ be the normal form of $w$ in $W$ for some
  identification of $W$ and $T \rtimes W_0$ and let $U = \mov(u)$. By
  Propositions~\ref{prop:translated move-set} and~\ref{prop:ell-recog}, 
  $\mov(w) = \lambda + U$, $\mov(w_e) = U$, and
  $e = \dim(U)$.  Let $m = \dim(w) = d+e$.  By definition, there
  exists an $m$-dimensional root subspace $M$ containing the affine
  subspace $\mov(w)$.  By Lemma~\ref{lem:sep}, $M$ must contain both
  $\lambda$ and $U$.  Since $M$ has a basis of roots, there exists a
  relative root basis $\{\alpha_1, \alpha_2, \ldots, \alpha_d\}$ for
  $M$ over $U$, i.e., $M = \spn(\{\alpha_1, \ldots, \alpha_d\}) \oplus
  U$ (and in particular, $\spn(\{\alpha_1, \ldots, \alpha_d\}) \cap U
  = \{0\}$).

  For each $i$, let $r_i$ be any reflection with root $\alpha_i$.
  Consider the element $v=w r_1 r_2 \cdots r_d$.  Iteratively applying
  Proposition~\ref{prop:move-add} one reflection at a time shows that
  $\mov(v) = \spn(\{\alpha_1, \ldots, \alpha_d\}) + \mov(w)$.  Since
  $\mov(w) = \lambda + U$, it follows that $\mov(v) = \spn(\{\alpha_1,
  \ldots, \alpha_d\}) + \lambda + U = \lambda + M = M$.  Since $M$ is
  a linear subspace of $V$, $v$ is elliptic and $\lr(v) = \dim(M) =
  d+e$.  Since $w = v r_d \cdots r_2 r_1$, $\lr(w) \leq \lr(v) + d =
  (d+e) + d = 2d+e$, as claimed.
\end{proof}

Together Propositions~\ref{prop:thmA-lb} and~\ref{prop:thmA-ub} prove
Theorem~\ref{main:dim}.

\subsection{Proof of Theorem~\ref{main:fact}}\label{subsec:B}

In this subsection, we prove Theorem~\ref{main:fact} using a recent
technical result about inefficient factorizations in spherical Coxeter
groups due to the first author and Vic Reiner
\cite[Corollary~1.4]{LewisReiner-16}.

\begin{prop}[Inefficient factorizations]\label{prop:inefficient}
  Let $W_0$ be a spherical Coxeter group and let $w$ be an element of
  $W_0$.  If $\lr(w) = \ell$, then every factorization of $w$ into $k$
  reflections lies in the Hurwitz orbit of some factorization $w = r_1
  r_2 \cdots r_k$ where $r_1 = r_2$, $r_3 = r_4$, \ldots $r_{k−\ell−1}
  = r_{k−\ell}$, and $r_{k−\ell+1} \cdots r_k$ is a minimum-length
  reflection factorization of $w$.
\end{prop}

We use this proposition about inefficient factorizations of an element
in a spherical Coxeter group to find a way to rewrite an efficient
factorization of an element in an affine Coxeter group into a
particular form.  This is our second main result.

\setcounter{mainthm}{\mainfact}
\begin{mainthm}[Factorization]
  Let $W$ be an affine Coxeter group.  For any element $w\in W$ there
  is a translation-elliptic factorization $w = t_\lambda u$ such that
  $\lr(t_\lambda) = 2d(w)$ and $\lr(u) = e(w)$.  
  In particular, $\lr(w) = \lr(t_\lambda) + \lr(u)$ for
  this factorization of $w$.
\end{mainthm}

\begin{proof}
  Let $k = 2d +e = \lr(w)$ and let $w = r_1' r_2' \cdots r_k'$ be a
  minimum-length reflection factorization of $w$.  The projection
  $p\colon W \onto W_0$ sends this efficient factorization of $w \in
  W$ to a (not necessarily efficient) factorization $w_e = p(w) =
  p(r_1') p(r_2') \cdots p(r_k')$ of $w_e \in W_0$.  By Carter's lemma
  (see Remark~\ref{rem:prior}) and Proposition~\ref{prop:ell-recog},
  minimum-length reflection factorizations of $w_e$ all have length $e
  = e(w)$.

  By Proposition~\ref{prop:inefficient}, there is a sequence of
  Hurwitz moves on the given $W_0$-factorization that produces a
  factorization of a special form.  The exact sequence of Hurwitz
  moves applied to the factorization in $W_0$ can be mimicked on the
  original factorization in $W$; since the Hurwitz action is easily
  seen to be compatible with $p$, the result is an $R$-factorization
  $w = r_1 r_2 \cdots r_k$ of $w$ such that $p(r_1) = p(r_2)$, $p(r_3)
  = p(r_4)$, \ldots, $p(r_{2d-1}) = p(r_{2d})$, and $p(r_{2d+1})
  \cdots p(r_k)$ is a minimum-length reflection factorization of
  $p(w)$.  This means that $r_1$ and $r_2$ are reflections through
  parallel hyperplanes, and so $t_{\lambda_1} = r_1 r_2$ is a
  (possibly trivial) translation. Similarly $t_{\lambda_i} = r_{2i -
    1} r_{2i}$ is a translation for $i = 1, \ldots, d$.  Thus the
  product of these first $2d$ reflections is a translation $t_\lambda
  = t_{\lambda_1} t_{\lambda_2} \cdots t_{\lambda_{d}}$ and
  $\lr(t_\lambda) \leq 2d$.

  Let $u = r_{2d+1} \cdots r_k$ be the product of the remaining
  $e$ reflections, so that $w = t_\lambda u$.  
  Because the given factorization of $u$ projects to a
  minimum-length reflection factorization of $u_e = w_e$ in $W_0$, the
  roots of $p(r_{2d+1})$, \ldots, $p(r_k)$ must be linearly
  independent (Lemma~\ref{lem:red-ell}), which means that the same is
  true for $r_{2d+1}$, \ldots, $r_k$ in $W$.  By Lemma~\ref{lem:red-ell},
  $u$ is elliptic and $\lr(u) = e$.  Finally, by Theorem~\ref{main:dim}
  and the triangle inequality (Remark~\ref{rem:lr-basic}), we must have 
  $\lr(t_\lambda) = 2d$.
\end{proof}

It is important to note that this theorem does \emph{not} assert that
there exists an identification of $W$ with $T \rtimes W_0$ such that
the corresponding normal form $w = t_\lambda u$ satisfies the
conclusion of the theorem.  In fact this stronger assertion is
demonstrably false.

\begin{example}[Normal forms are insufficient]\label{ex:insuff}
  In all irreducible affine Coxeter groups other than the affine
  symmetric groups, there is a maximal parabolic subgroup $W'$ that is not
  isomorphic to $W_0$.  Let $w$ be a Coxeter element for one of these
  alternative maximal parabolic subgroups. 
  
  Being a Coxeter element of a spherical Coxeter group, $w$ has a
  unique fixed point; denote it by $x'$. Reflections across the
  hyperplanes passing through $x'$ generate $W'$. On the other hand,
  reflections across the hyperplanes passing through any choice of
  origin used in the construction of $W$ (i.e., the $r_{\alpha,0}$
  reflecting across the $H_{\alpha,0}$ as in Definitions
  \ref{def:affine} and \ref{def:sub-quo}) must generate a group
  isomorphic to $W_0$. As $W'$ is not isomorphic to $W_0$, this means
  $x'$ cannot be the origin for any such identification of $W$ with $T
  \rtimes W_0$.
  
  Now fix an identification of $W$ and $T \rtimes W_0$ having origin
  $x\neq x'$, and consider the corresponding normal form $w =
  t_\lambda u$. By Remark~\ref{rem:max-ell}, $\lr(w) = \lr(u) =
  n$. However, since $w(x)\neq x = u(x)$, we have $t_{\lambda}\neq 1$,
  i.e., the normal form for $w$ includes a nontrivial translation.
  Thus, for this element $w$, $\lr(w) < \lr(t_\lambda) + \lr(u)$.
  
  Figure \ref{fig:B2notnormal} provides an illustration.  Here we see
  a portion of the affine line arrangement for $B_2$. The nodes in
  black correspond to all possible identifications of an origin to
  form the semidirect product.
  
  Let $r$ and $s$ be the reflections across the lines indicated in
  bold, and let $w=rs$. Let $x'$ denote the point of intersection for
  these lines. (The reflections $r$ and $s$ generate the maximal
  parabolic subgroup of type $A_1\times A_1$.) We have $\lr(w) =
  \lr(w_e)=2$. However, if we choose an origin $x\neq x'$ and write
  $w=t_{\lambda}u$ with respect to this origin, then the motion $u$
  will be a $\pi$-rotation about $x$ which requires at least $2$
  reflections on its own, i.e. $\lr(u) = 2$.  But since $x$ and $x'$
  are not equal and $w$ does not fix $x$, $t_\lambda$ cannot be
  trivial.  Thus $\lr(t_\lambda) >0$ and $\lr(w) < \lr(t_\lambda) +
  \lr(u)$.
\end{example}
  
\begin{figure}
  \begin{tikzpicture}[>=stealth,baseline=2.5cm,scale=1]
    \foreach \x in {1,...,10}{
      \draw (\x,.8)--(\x,6.2);
    }
    \foreach \x in {1,...,6}{
      \draw (.8,\x)--(10.2,\x);
    }
    \foreach \x in {1,...,3}{
      \draw (.8,2*\x+.2)--(2*\x+.2,.8);
      \draw (.8,2*\x-.2)--(7.2-2*\x,6.2);
      \draw (2*\x+3.8,.8)--(10.2,7.2-2*\x);
    }
    \foreach \x in {1,...,2}{
      \draw (2*\x+.8,6.2)--(2*\x+6.2,.8);
      \draw (2*\x+4.8,6.2)--(10.2,2*\x+.8);
      \draw (2*\x+5.2,6.2)--(2*\x-.2,.8);
    }
    \draw[very thick] (6,0)--(6,7) node[right] {$r$};
    \draw[very thick] (0,4)--(11,4) node[right] {$s$};
    \draw (6,4) node[inner sep=2,circle,draw=black,fill=white] {}
    node[inner sep = 1,xshift=8pt,yshift=8pt] {$x'$}; 
    \foreach \x in {1,...,5}{
      \foreach \y in {1,2,3}{
        \draw (2*\x-1,2*\y) node[circle,inner sep=2, fill=black] {};
        \draw (2*\x,2*\y-1) node[circle,inner sep=2, fill=black] {};
      }
    }
  \end{tikzpicture}
  \caption{The element $w=rs$ has $\ell_R(w) = \ell_R(w_e)=2$. The
    origin can only be identified with the black nodes, whereas $x'$
    is the unique fixed point of $w$. Thus every normal form for $w$
    has a nontrivial translation.}\label{fig:B2notnormal}
\end{figure}
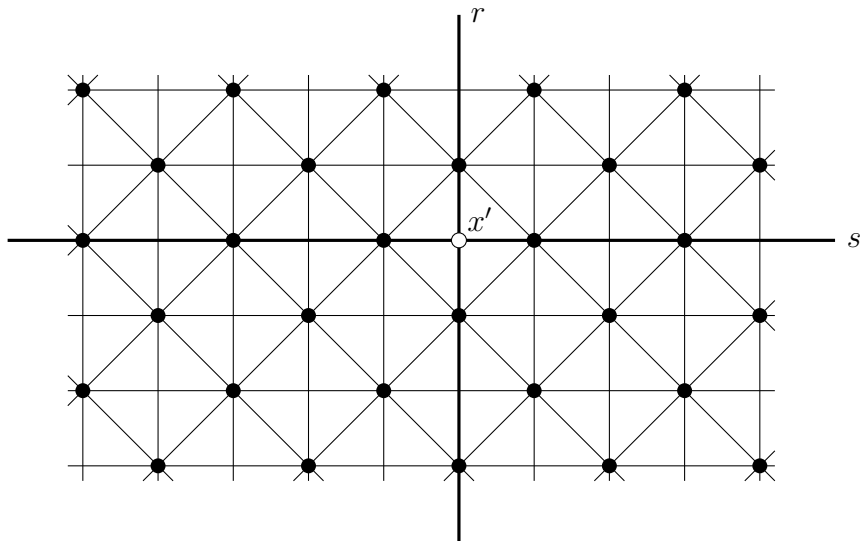

In an affine symmetric group, all maximal parabolic subgroups are
isomorphic and every vertex (maximal intersection of hyperplanes)
in the hyperplane arrangement can play the role of the origin in an
identification with a semidirect product.  This immediately
establishes the following result.

\begin{corollary}[Affine symmetric normal form]
  Let $W = \affS_n$ be the affine symmetric group.  For each element
  $w \in W$ there is an identification of $W$ and $T \rtimes W_0$ so
  that $w$ has normal form $w = t_\lambda u$ and $\lr(w) =
  \lr(t_\lambda) + \lr(u)$.
\end{corollary}

\begin{proof}
  Let $w = t_\lambda u$ be the factorization from
  Theorem~\ref{main:fact}.  Choose an identification of $W$ and
  $T \rtimes W_0$ so that the role of the origin is played by one
  of the vertices that is fixed by $u$.
\end{proof}

\section{Local statistics}\label{sec:gen-fns}

In a spherical Coxeter group $W = W_0$, Shephard and Todd
\cite[Thm.~5.3]{ShephardTodd} showed that the generating function
\[
f_0(t) = \sum_{u \in W_0} t^{\ell_R(t)}
\]
has a particularly nice form:
\begin{equation}\label{eq:ST}
f_0(t) = \prod_{i = 1}^n (1 + e_i t),
\end{equation}
where the numbers $e_i$ are positive integers called the
\emph{exponents} of $W_0$.  In earlier work \cite{McPe-11}, the second
and third authors asked whether there are similarly nice generating
functions associated to an affine Coxeter group.  In this section, we
explore this question.

For an affine Coxeter group $W$, reflection length is bounded and
$|W|$ is infinite, so the naive generating function is not defined.  A
natural fix is to consider only a finite piece of $W$.

\begin{defn}[Local generating function]
  Given an element $\lambda$ of the coroot lattice $L(\Phi^\vee)$,
  define the bivariate generating function
  \[
  f_{\lambda}(s, t)
  = \sum_{u \in W_0} s^{d(t_\lambda u)} \cdot t^{e(t_\lambda  u)}
  = \sum_{u \in W_0} s^{\dim(t_\lambda  u)} \cdot (t/s)^{\dim(u)}
  \] 
  that tracks the statistics of differential and elliptic dimension.
  By Theorem~\ref{main:dim}, we have
  \[
  f_{\lambda}(t^2, t) = \sum_{u \in W_0} t^{\lr(t_\lambda \cdot u)}.
  \]
  By mild abuse of notation, we let $f_{\lambda}(t) =
  f_{\lambda}(t^2,t)$ denote this local reflection length generating
  function.
\end{defn}

The term ``local'' here makes sense geometrically, once the elements
of $W$ have been identified with alcoves in the reflecting hyperplane
arrangement for $W$ as in Remark \ref{rem:alcoves}. The alcoves
neighboring the origin (each of which is identified with a unique
element of $W_0$) form a $W_0$-invariant polytope $P$. The set $\{
t_\lambda u \,|\, u \in W_0\}$ corresponds to the set the of alcoves
in $t_\lambda \cdot P$, i.e., those alcoves neighboring $\lambda$. In
Figure \ref{fig:graylocal} we have shaded the alcoves according to
reflection length. In each example, the identity element is identified
with the black alcove, and lighter colored cells have greater
reflection length. The coroots are highlighted in white.

\begin{figure}
  \[
  \begin{array}{cc}
    \includegraphics[trim={2cm 2cm 3cm 2cm}, clip, scale=.75]{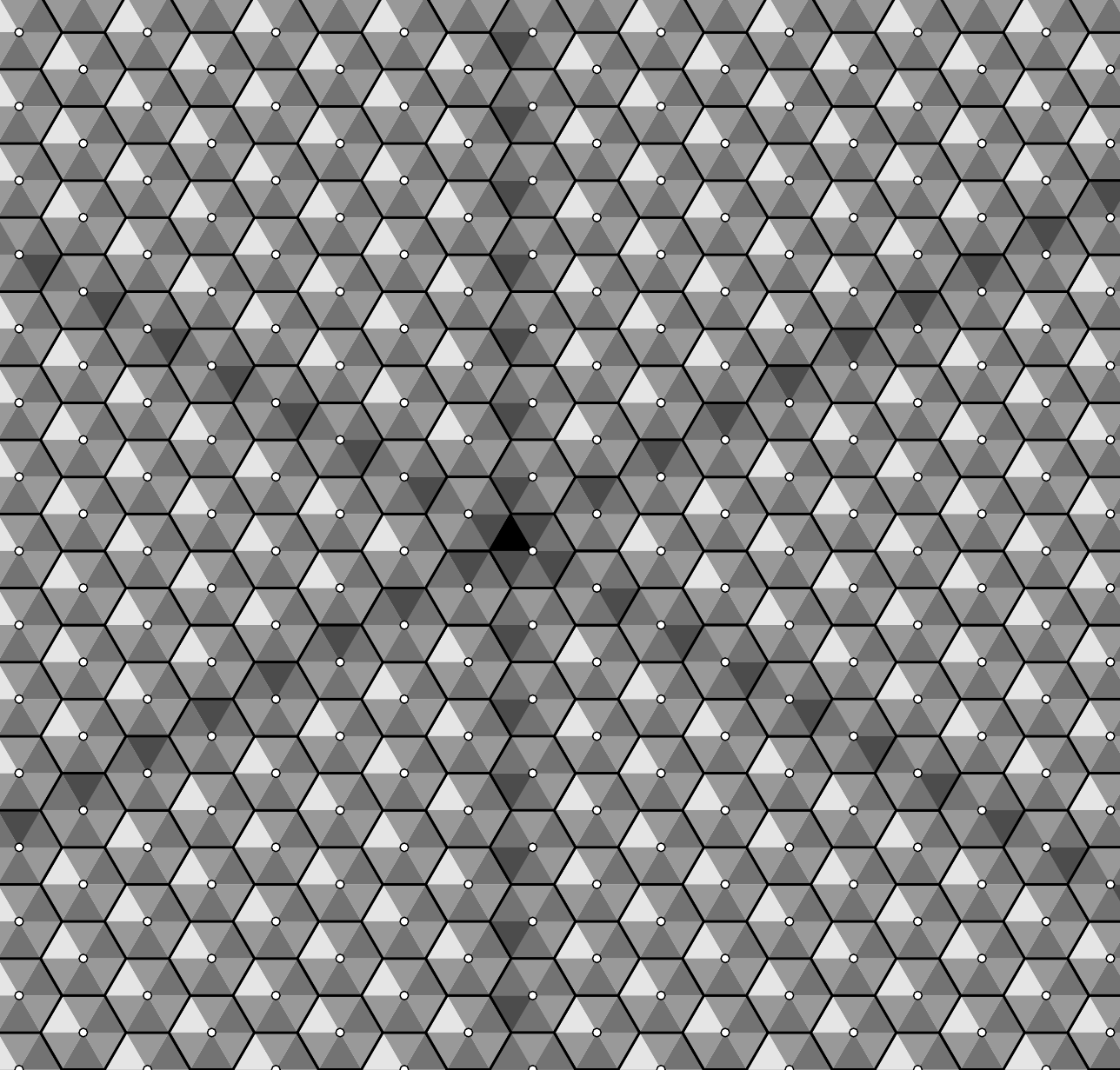}
    &
    \includegraphics[trim={2cm 3cm 2.5cm 3.5cm}, clip, scale=.8]{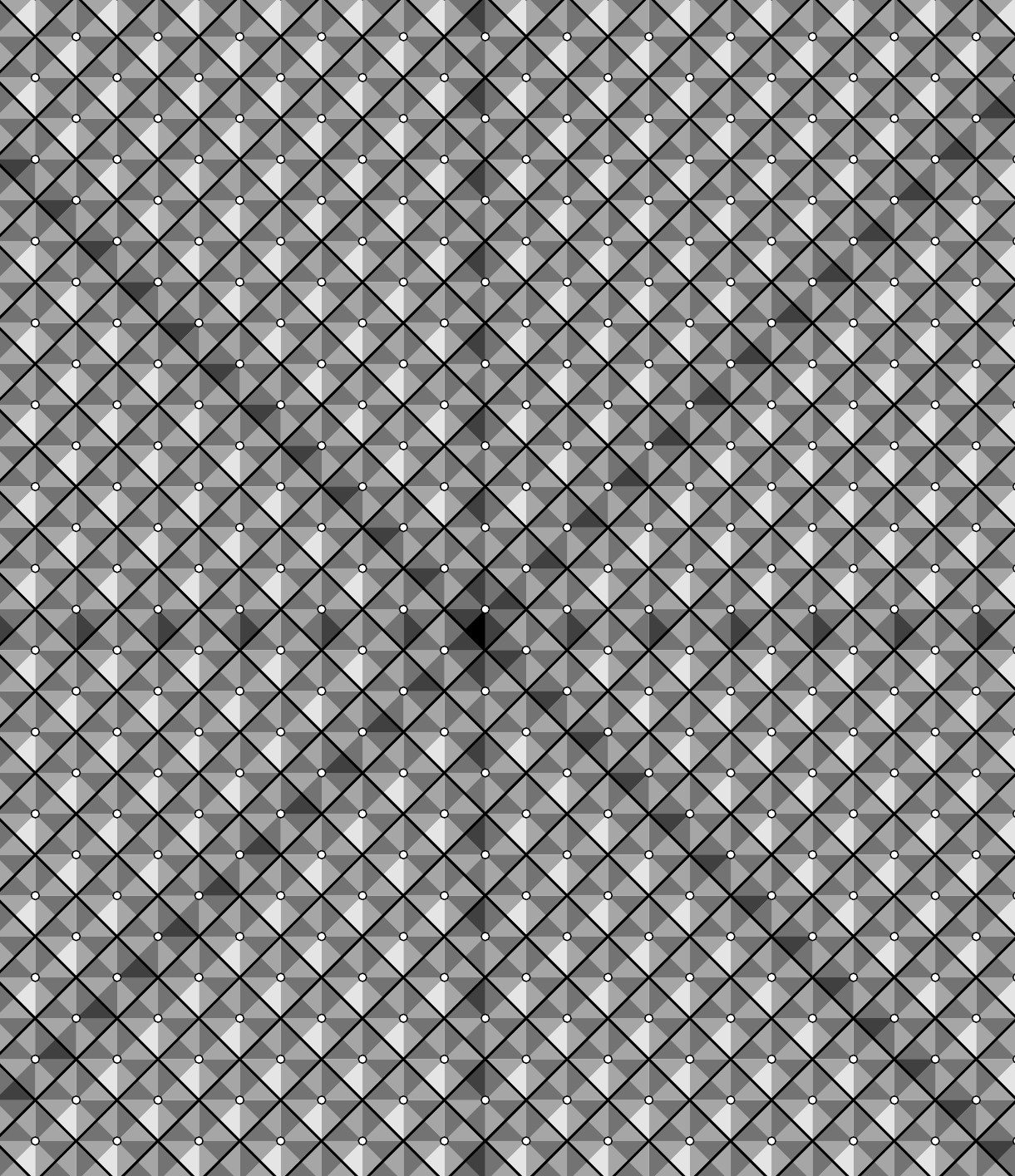}
    \\
    (a) & (b)
  \end{array}
  \]
  \[
  \begin{array}{c}
    \includegraphics[trim={2cm 2cm 3cm 2cm}, clip, scale=.75]{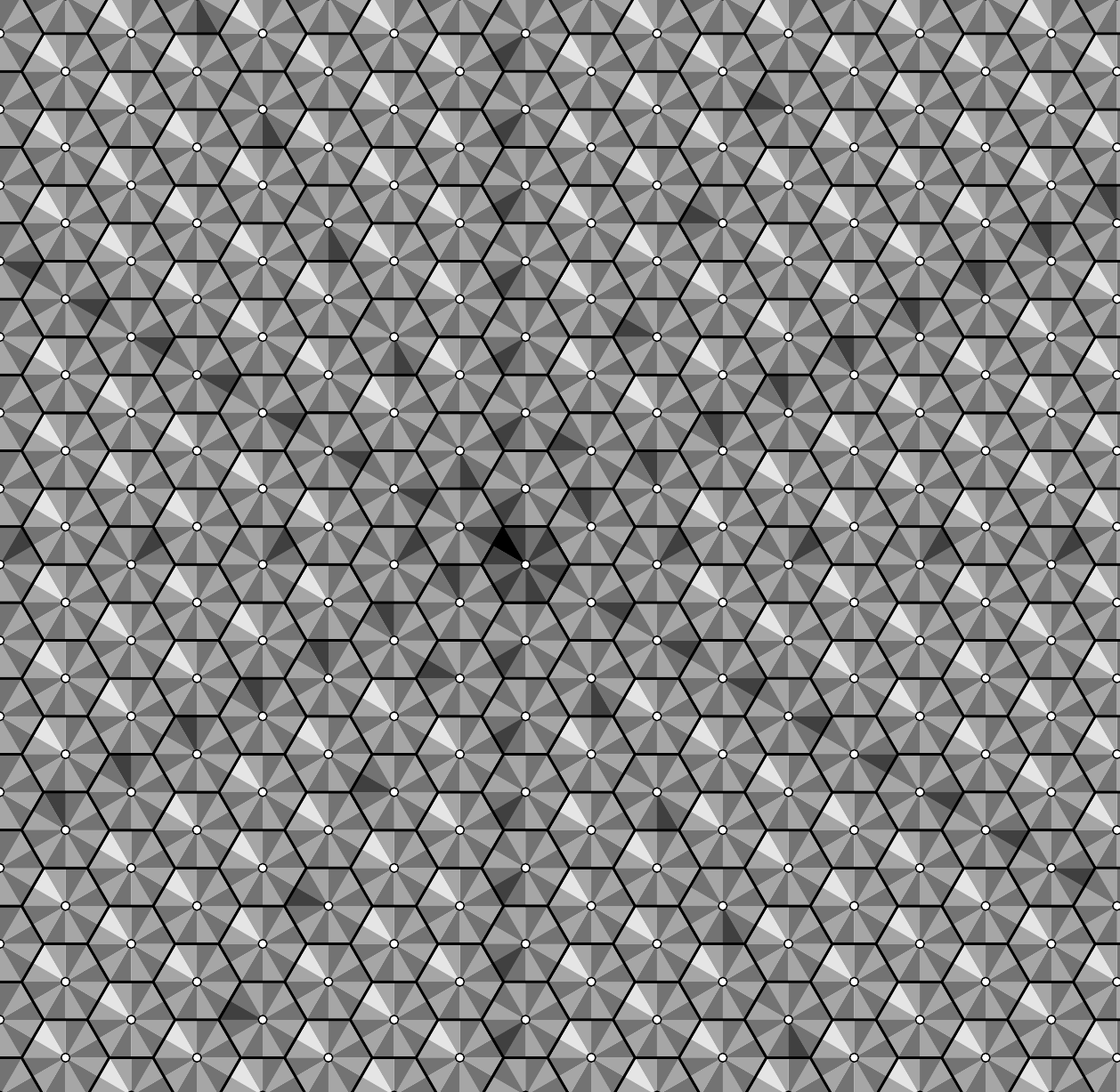}
    \\
    (c)
  \end{array}
  \]
  \caption{Reflection length in the affine Coxeter groups of (a) type
    $A_2$, (b) type $B_2$, and (c) type $G_2$. Lighter colors
    correspond to greater reflection length. White nodes correspond to
    translations.}\label{fig:graylocal}
\end{figure}

We collect here some easy facts about the local generating function.

\begin{prop}[Properties of local generating functions]\label{prop:simple gf facts}
  Let $\lambda$ be a vector in the coroot lattice $L(\Phi^\vee)$.
  \begin{enumerate}[label=(\roman*)]
  \item (The origin) If $\lambda=0$, then $f_{\lambda}(s, t) = f_0(t)$
    is the generating function for reflection length in $W_0$.
  \item (Generic) If $\lambda$ is generic, i.e., if it is not a member
    of any proper root subspace of $V$, then $f_\lambda(s, t) = s^{n}
    f_0(t/s) = \prod_{i = 1}^n (s + e_i t)$, where the $e_i$
    are the exponents of $W_0$.
  \item (Permutations) If $\lambda$ and $\lambda'$ belong to the same
    $W_0$-orbit (that is, there is some $w \in W_0$ such that
    $\lambda' = w(\lambda)$), then $f_{\lambda}(s, t) =
    f_{\lambda'}(s, t)$.
  \end{enumerate}
\end{prop}

\begin{proof}
  If $\lambda = 0$ then $t_\lambda$ is the identity, and so (i)
  follows immediately from the definitions.

  For (ii), since $u$ is elliptic we have $\lambda \in \mov(t_\lambda
  u)$.  Thus, since $\lambda$ is generic, $\mov(t_\lambda u)$ is not
  contained in any proper root subspace of $V$, and so $\dim(t_\lambda
  u) = n$.  Substituting this into the definition of $f_\lambda$ gives
  the first equality.  The second equality follows from Shephard and 
  Todd's result.

  For (iii), suppose that $\lambda' = w(\lambda)$.  Then $t_{\lambda'}
  = w t_{\lambda} w^{-1}$ and so for every $u$ in $W_0$ we have
  $t_{\lambda'}u = w( t_\lambda (w^{-1} u w)) w^{-1}$.  Conjugation of
  group elements by $w$ acts on move-sets as multiplication by $w$.
  Since $w \in W_0$, this preserves root dimensions, and so
  $\dim(t_{\lambda'} u) = \dim(t_{\lambda} (w^{-1}uw))$.  Finally, as
  $u$ runs over the spherical group $W_0$, $w^{-1}uw$ does as well,
  and $\dim(w^{-1}uw) = \dim(u)$, so the joint distribution of
  dimensions is the same over both sets.
\end{proof}

To get a feel for the local generating functions, Table
\ref{tab:rank2gf} lists the different local generating functions for
types $A_2$, $B_2$, and $G_2$. In Figure~\ref{fig:colorlocal} we
revisit the alcove picture, but instead of individual alcoves, we have
colored the translates $t_\lambda \cdot P$ according to
$f_\lambda(t)$. In Figure \ref{fig:faces} we see affine $A_3$ with the
translates $t_\lambda\cdot P$ colored according to $f_\lambda(t)$. In
this case there are six different local generating functions, as
listed in Table \ref{tab:A3gf}. (Only five are visible in Figure
\ref{fig:faces} as the origin is hidden.) The important thing to know
about the notation in Table \ref{tab:A3gf} is that $\alpha_1$ and
$\alpha_3$ are orthogonal to each other, whereas neither is orthogonal
to $\alpha_2$.

From the pictures it appears that the local generating functions line
up along faces of a hyperplane arrangement. The faces of the
arrangement are intersections of maximal root subspaces, and if two
coroots lie in the same face, they have the same local generating
function. However, note that while some of these intersections are
themselves root subspaces, not all of them are. For example, the white
balls in Figure \ref{fig:faces} lie along lines that are not of the
form $\R\alpha$ for any root $\alpha$.

\begin{table}[t]
  \[
  \begin{array}{c}
    A_2 \\
    \hline
    \begin{array}{cc|c}
      & \lambda & f_\lambda(s,t)\\
      \hline
      \begin{tikzpicture} \draw(0,0) node[circle,inner sep=2, fill=blue,draw=black] {}; \end{tikzpicture} & 0
      & (1+t)(1+2t) \\
      \begin{tikzpicture} \draw(0,0) node[circle,inner sep=2, fill=white!60!blue,draw=black] {}; \end{tikzpicture}  & \alpha^{\vee} & (s+t)(1+2t) \\
      \begin{tikzpicture} \draw(0,0) node[circle,inner sep=2, fill=white!90!blue,draw=black] {}; \end{tikzpicture}  & \text{generic} & (s+t)(s+2t)
    \end{array}
  \end{array}
  \hspace{.5cm}
  \begin{array}{c}
    B_2 \\
    \hline
    \begin{array}{cc|c}
      & \lambda & f_\lambda(s,t)\\
      \hline
      \begin{tikzpicture} \draw(0,0) node[circle,inner sep=2, fill=blue,draw=black] {}; \end{tikzpicture} & 0 & (1+t)(1+3t) \\
      \begin{tikzpicture} \draw(0,0) node[circle,inner sep=2, fill=white!60!blue,draw=black] {}; \end{tikzpicture} & \alpha^{\vee} & (s+t)(1+3t) \\
      \begin{tikzpicture} \draw(0,0) node[circle,inner sep=2, fill=white!90!blue,draw=black] {}; \end{tikzpicture} & \text{generic} & (s+t)(s+3t)
    \end{array}
  \end{array}
  \]
  \[
  \begin{array}{c}
    G_2 \\
    \hline
    \begin{array}{cc|c}
      &\lambda & f_\lambda(s,t)\\
      \hline
      \begin{tikzpicture} \draw(0,0) node[circle,inner sep=2, fill=blue,draw=black] {}; \end{tikzpicture} & 0 & (1+t)(1+5t) \\
      \begin{tikzpicture} \draw(0,0) node[circle,inner sep=2, fill=white!60!blue,draw=black] {}; \end{tikzpicture} & \alpha^{\vee} & (s+t)(1+5t) \\
      \begin{tikzpicture} \draw(0,0) node[circle,inner sep=2, fill=white!90!blue,draw=black] {}; \end{tikzpicture} & \text{generic} & (s+t)(s+5t)
    \end{array}
  \end{array}
  \]
  \caption{Local generating functions for affine $A_2$, $B_2$, and $G_2$.}\label{tab:rank2gf}
\end{table}

\begin{figure}
  \[
  \begin{array}{cc}
    \includegraphics[scale=.5]{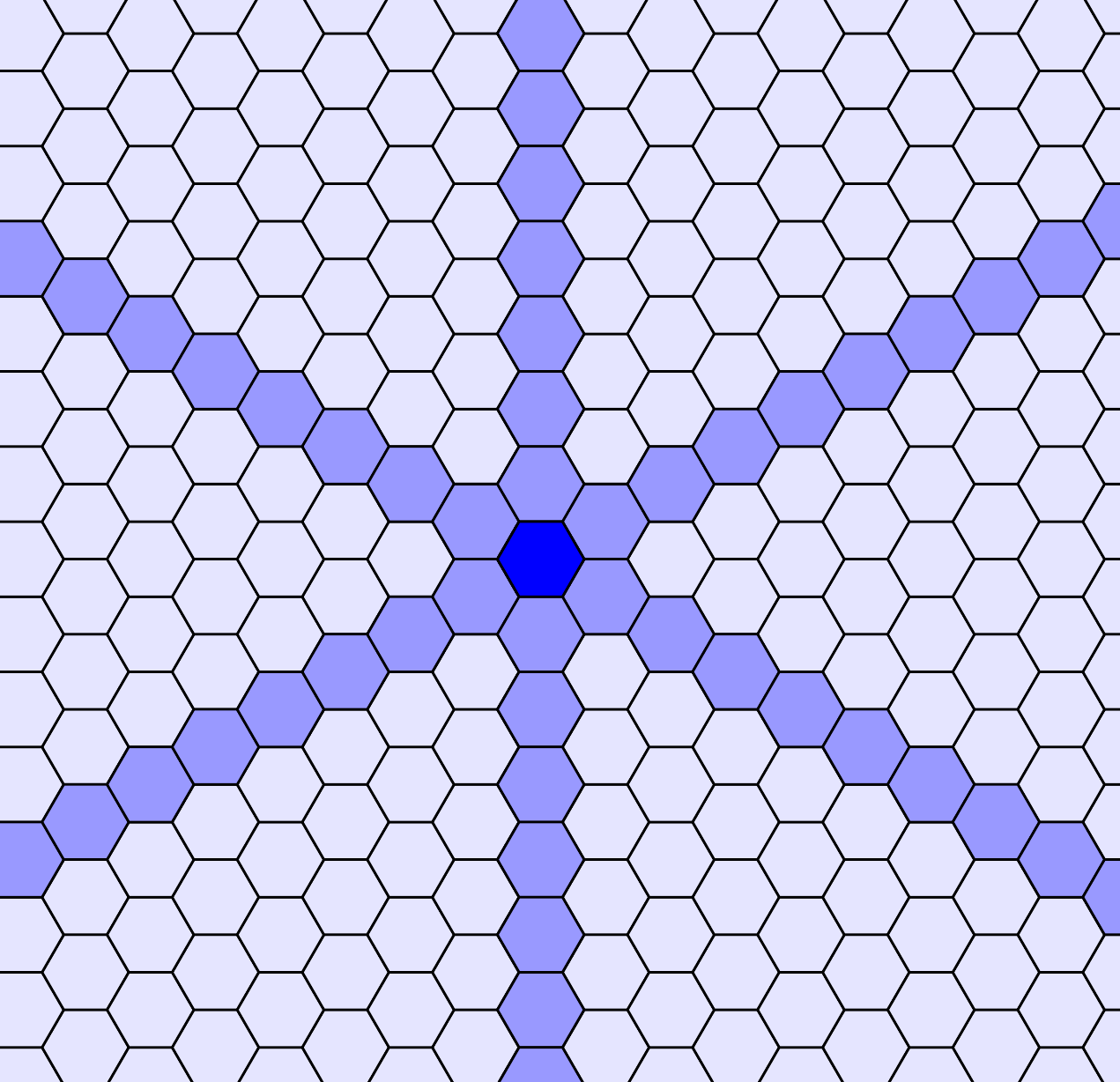}
    &
    \includegraphics[scale=.5]{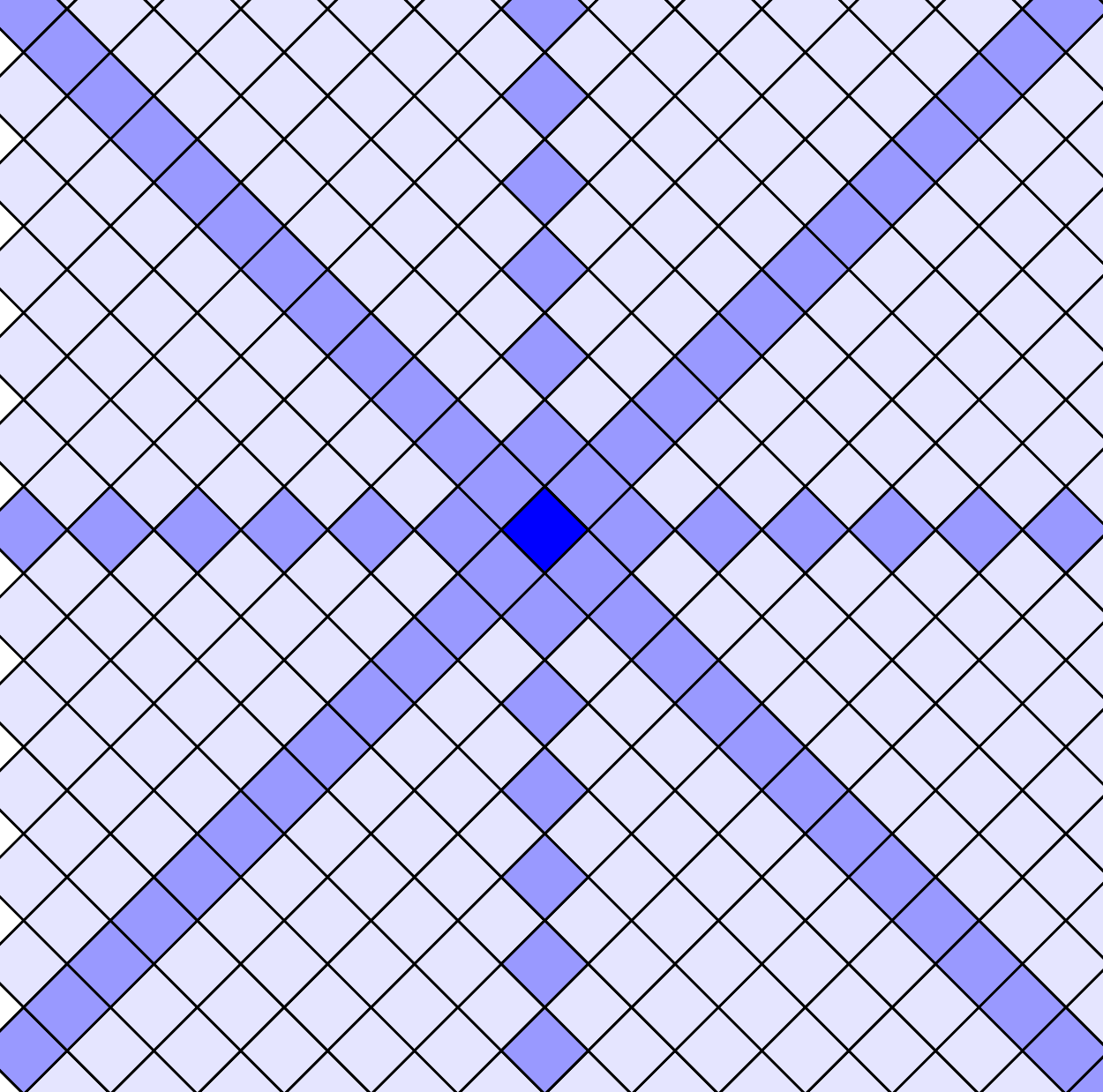}
    \\
    (a) & (b)
  \end{array}
  \]
  \[
  \begin{array}{c}
    \includegraphics[scale=.5]{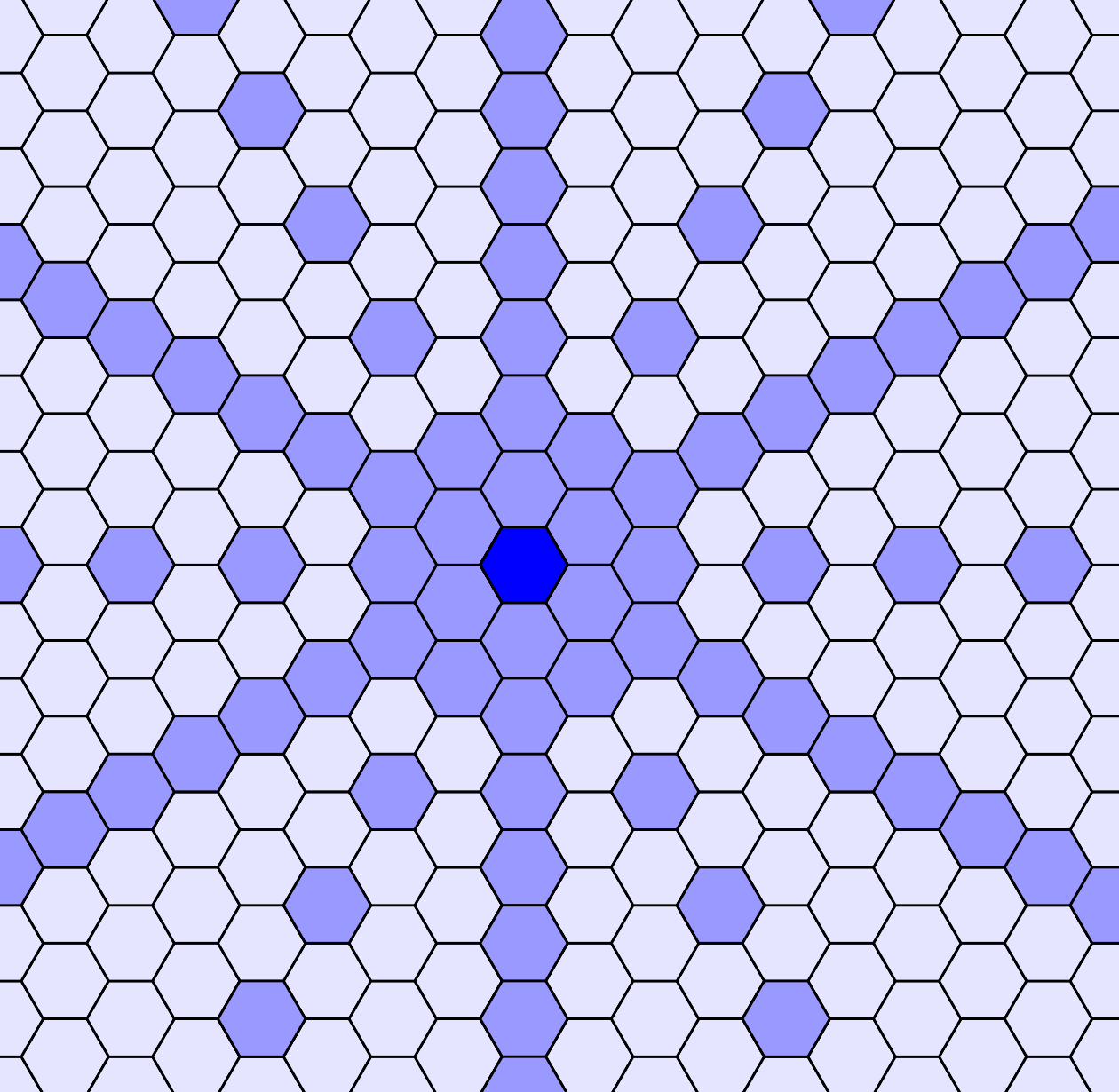}
    \\
    (c)
  \end{array}
  \]
  \caption{The translates $t_\lambda\cdot P$ in (a) type $A_2$, (b)
    type $B_2$, and (c) type $C_2$, colored according to the local
    distribution of reflection length.}\label{fig:colorlocal}
\end{figure}

\begin{figure}
  \begin{center}
    \includegraphics[scale=.7]{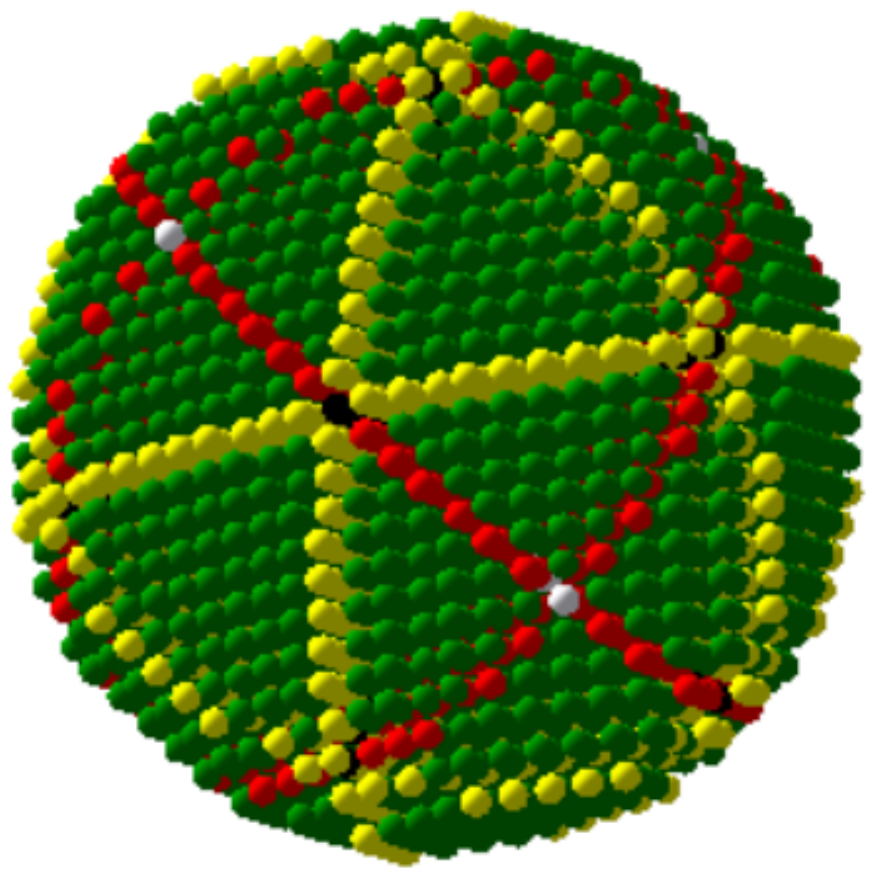}
  \end{center}
  \caption{The translates $t_\lambda\cdot P$ in type $A_3$ colored
    according to the local distribution of reflection length.}
  \label{fig:faces}
\end{figure}

\begin{table}[t]
  \[
  \begin{array}{c}
    A_3 \\
    \hline
    \begin{array}{cc|c}
      &\lambda & f_\lambda(s,t)\\
      \hline
      \begin{tikzpicture} \draw(0,0) node[circle,inner sep=2, fill=orange,draw=black] {}; \end{tikzpicture} & 0 & (1+t)(1+2t)(1+3t) \\
      \begin{tikzpicture} \draw(0,0) node[circle,inner sep=2, fill=black,draw=black] {}; \end{tikzpicture} & \alpha^{\vee} & (s+t)(1+2t)(1+3t) \\
      \begin{tikzpicture} \draw(0,0) node[circle,inner sep=2, fill=white!90!black, draw=black] {}; \end{tikzpicture} &  \alpha_1^\vee + \alpha_3^\vee & 2t^2 + 6t^3 + 4st+9st^2 + s^2+2s^2t \\
      \begin{tikzpicture} \draw(0,0) node[circle,inner sep=2, fill=yellow,draw=black] {}; \end{tikzpicture} & \text{generic span of } \alpha_1^{\vee}, \alpha_2^{\vee} & (s+t)(s+2t)(1+3t) \\
      \begin{tikzpicture} \draw(0,0) node[circle,inner sep=2, fill=red,draw=black] {}; \end{tikzpicture} & \text{generic span of } \alpha_1^{\vee}, \alpha_3^{\vee} & (s+t)(t+6t^2+s+4st) \\
      \begin{tikzpicture} \draw(0,0) node[circle,inner sep=2, fill=green!60!black,draw=black] {}; \end{tikzpicture} & \text{generic} & (s+t)(s+2t)(s+3t) \\
    \end{array}
  \end{array}
  \]
  \caption{Local generating functions for affine $A_3$. Here
    $\alpha_1$ and $\alpha_3$ are any two orthogonal roots, while
    $\alpha_2$ and $\alpha_1$ are not orthogonal.}\label{tab:A3gf}
\end{table}

We make the phenomenon precise here.

\begin{thm}[Equality of local generating functions]\label{thm:faces}
  Suppose that $\lambda$ and $\mu$ are two elements of the coroot
  lattice $L$.  Suppose furthermore that $\lambda$
  and $\mu$ belong to the same collection of root subspaces.
  Then $f_\lambda(s, t) = f_\mu(s, t)$.
\end{thm}

\begin{proof}
  In fact, something stronger is true: if $\lambda$ and $\mu$ are as
  described, then $\dim(t_\lambda u) = \dim(t_\mu u)$ for every $u$ in
  $W_0$.  Indeed, by Lemma~\ref{lem:sep}, for $u \in W_0$, a root
  subspace $U$ contains $\mov(t_\lambda u)$ if and only if it contains
  $\lambda$ and $\mov(u)$. But by hypothesis, every such $U$ contains
  $\mu$ and $\mov(u)$, and so contains $\mov(t_\mu u)$.  Then the
  result follows immediately from the definition of $f_\lambda$.
\end{proof}

Unfortunately, while Theorem \ref{thm:faces} and Proposition
\ref{prop:simple gf facts} imply bounds on the number of local
generating functions in terms of the number of $W_0$-orbits of
intersections of root subspaces, it is probably intractable to compute
all $f_\lambda(s,t)$, or even all $f_\lambda(t)$, in general. We show
in Appendix \ref{sec:null-comp} that computing $d(t_\lambda)$ for an
element $\lambda$ of the type $A_n$ coroot lattice is essentially
equivalent to the NP-complete problem \texttt{SubsetSum}.

A different approach to understanding the distribution of reflection
length would be to introduce a new statistic that grows with $\lambda$
and take a bivariate generating function, either over the whole group
$W$ or over the elements with fixed elliptic part (that is, over a
coset of the translations).  Thus, for a given element $u \in W_0$ one
could consider the generating function
\[
g_{u}(q,t) = \sum_{\lambda \in L} t^{d(t_{\lambda} u)}q^{\operatorname{stat}(\lambda, u)}
\]
for some statistic ``$\operatorname{stat}$'', e.g., usual the Coxeter
length of $t_\lambda$. By Theorem~\ref{main:dim}, the corresponding
generating function for reflection length is $t^{\dim(u)} g_u(q,
t^2)$.

If $u$ has maximal reflection length in $W_0$, e.g., if $u$ is a
Coxeter element, then by Remark \ref{rem:max-ell} we have
$d(t_{\lambda} u)=0$, and this generating function simplifies to
$\sum_{\lambda \in L} q^{\operatorname{stat}(\lambda, u)}$. However,
at this point we have no good candidate for
$\operatorname{stat}(\lambda, u)$ and have made little progress in
this direction.

\section{Affine symmetric groups}\label{sec:aff-sym}

In this section, we restrict our attention to the affine symmetric
groups and give simple combinatorial descriptions of the statistics
$d(w)$ and $e(w)$ used to define reflection length.  This provides an
affine analog of the formula for the symmetric group given by D\'enes
\cite{Denes59}.

Throughout this part of the article, we can safely relax our
distinction between ``points'' in euclidean space and ``vectors'' in a
vector space.  To this end, fix a euclidean vector space $\R^n$ with a
fixed ordered orthonormal basis $e_1, e_2,\ldots, e_n$.

\subsection{Permutations and affine permutations}\label{sec:perms}

In this section, we describe the symmetric group and affine symmetric
group from both the geometric and combinatorial perspectives.

\begin{defn}[Permutations]\label{def:perms}
  A \emph{permutation} of the set $[n]$ is a bijection $\pi\colon [n]
  \to [n]$ and the set of all permutations under composition is the
  \emph{symmetric group} $\Sym_n$.  A permutation $\pi$ may be
  represented in \emph{one line notation} as the sequence $[\pi(1),
    \ldots, \pi(n)]$ of its values.
\end{defn}

\begin{defn}[Root system]\label{def:affS-roots}
  For every permutation $\pi \in \Sym_n$, there is an isometry
  $u_\pi\colon \R^n \to \R^n$ that sends $(v_1, \ldots, v_n)$ to
  $(v_{\pi(1)}, \ldots, v_{\pi(n)})$.  All of these isometries fix the
  line $\R \one$ in direction $\one = (1, \ldots, 1)$.  Let $V$ be the
  orthogonal complement of this line, consisting of all vectors whose
  coordinate sum is $0$.  Then $\Sym_n$ is a spherical Coxeter group
  acting on $V$, as in Definition~\ref{def:spherical coxeter group}.
  The reflections are exactly the \emph{transpositions} that
  interchange two values $i, j \in [n]$ while fixing all others.  For
  the reflection that switches the $i$-th and $j$-th coordinates we
  select the vectors $\pm(e_i - e_j)$ as its roots, and we denote by
  $\Phi_n$ the full root system $\Phi_{n} = \{ e_i -e_j \mid i, j \in
  [n], i\neq j\}$.
\end{defn}

We record a few elementary properties of $\Phi_n$ without proof.

\begin{prop}[Coroots]\label{prop:coroots-affS}
  Let $\Phi_n$ be the root system for $\Sym_n$. Roots are the same as
  coroots, the root system spans $V$, and the lattice of (co)roots is
  the set of vectors in $V$ with integer coordinates.  In symbols,
  $\Phi_n = (\Phi_n)^\vee$, $\spn(\Phi_n) = V$ and $L(\Phi_n)
  = V \cap \Z^n$.
\end{prop}

Following Definition~\ref{def:affine}, one can use the root system
$\Phi_n$ to construct the \emph{affine symmetric group} $\affS_n$ as a
group of isometries of $V$.  By Definition~\ref{def:sub-quo}, every
\emph{affine permutation} $w\in \affS_n$ can be written as $w =
t_\lambda u_\pi$ where $u_\pi$ is the elliptic isometry of $\R^n$
indexed by a permutation $\pi \in \Sym_n$ and $t_\lambda$ is a
translation by a vector $\lambda$ in the $\Z$-span of $\Phi_n$, i.e.,
a vector with integer coordinates and coordinate sum $0$.  There is
also an alternative description of affine permutations, due originally
to Lusztig \cite{Lusztig83}, that is common in the combinatorics
literature (see, e.g., \cite{Eriksson98}).

\begin{defn}[Affine permutations as bijections]\label{rem:affS-comb}
  Let $w\colon \Z \to \Z$ be a bijection from the integers to the
  integers.  We call the map $w$ an \emph{affine permutation} of
  \emph{order} $n$ if it satisfies three conditions:
  \begin{enumerate}
    \item  $w(i+n)=n+w(i)$ for all $i \in \Z$,
    \item  $w(i) = w(j) \mod n$ if and only if $i = j \mod n$, and   
    \item $w(1) + w(2) + \cdots + w(n) = 1 + 2 + \cdots + n$.
  \end{enumerate}
  By the first condition, such periodic maps can be described in
  \emph{one line notation} by listing the $n$ values
  $[w(1),w(2),\ldots,w(n)] \in \Z^n$.  It is easy to check whether
  such a list of $n$ integers satisfies the second and third
  conditions.
\end{defn}

\begin{rem}[Isomorphisms]
  The isomorphisms between the combinatorial and geometric
  descriptions of $\affS_n$ go as follows.  Since every integer can be
  uniquely written as a number in $[n]$ plus a multiple of $n$, the
  one line notation $[w(1),w(2),\ldots,w(n)]$ can be uniquely written
  as $u + t$, where $u$ has all entries in $[n]$ and $t = n \cdot
  \lambda$ for some vector $\lambda \in \Z^n$.  By the second
  condition, $u$ is the one line notation of a permutation $\pi$ in
  $\Sym_n$, while by third condition, $\lambda$ has coordinate sum
  $0$.  The bijection $w$ with this one line notation is sent by the
  isomorphism to the element $t_\lambda u_\pi \in \affS_n$.
\end{rem}

\subsection{Partitions, cycles, and root arrangements}

In this section, we describe the move-sets and fixed spaces of
permutations in terms of their cycle structure.  For this purpose, it
is helpful to work with the language of set partitions.

\begin{defn}[Partitions]\label{def:partitions}
  A \emph{partition} $P$ of a set $A$ is a collection of pairwise
  disjoint nonempty subsets whose union is $A$.  The subsets in $P$
  are called \emph{blocks}, and the \emph{size} $|P|$ of $P$ is the
  number of blocks it contains.  Concretely, a partition of size $k$
  is a collection $P = \{B_1, B_2, \ldots, B_k\}$ with $\emptyset \neq
  B_i \subset A$ for all $i$, $B_i \cap B_j = \emptyset$ for $i \neq
  j$, and $B_1 \cup \cdots \cup B_k = A$.
\end{defn}

\begin{defn}[Cycles and partitions]\label{rem:cyc-part}
  For every permutation $\pi \in \Sym_n$, there is a partition
  $P(\pi)$ with $i$ and $j$ in the same block of $P(\pi)$ if and only
  if $i$ and $j$ are in the same orbit under the action of $\pi$.
  Concretely, $i$ and $j$ are in the same block of $P(\pi)$ if and
  only if there is an integer $\ell$ so that $\pi^\ell(i) = j$.  When
  $\pi$ is written in \emph{cycle notation}, each block of $P(\pi)$
  corresponds to the set of numbers contained one cycle of $\pi$.  For
  example, the permutation $\pi$ in $\Sym_6$ with one line notation
  $[4, 5, 1, 3, 2, 6]$ has cycle notation $\pi = (1,4,3)(2,5)(6)$ and
  associated partition $P(\pi) = \{\{1,3,4\},\{2,5\},\{6\}\}$.  Let
  $\cyc(\pi)$ be the set of cycles in $\pi$, so that $\size{\cyc(\pi)}
  = \size{P(\pi)}$ is the number of cycles.
\end{defn}

\begin{defn}[Partition lattice]\label{def:part_n}
  Let $P$ and $Q$ be two partitions of the same set $A$. We say that
  $P$ \emph{refines} $Q$ and write $P \leq Q$ if every block of $P$ is
  contained in some block of $Q$.  We denote by $\prt_A$ the set of
  all partitions of a fixed finite set $A$, ordered by
  refinement. This poset is a bounded, graded lattice, and two
  partitions are at the same height if and only if they have the same
  size.  The partition $P_\one = \{A\}$ with only one block is the
  unique maximum element.  The partition $P_\zero$ with $\size{A}$
  blocks, each containing a single element, is the unique minimum
  element.  When $A = [n]$, we write $\prt_n$ instead of $\prt_{[n]}$.
\end{defn}

The next few definitions describe the spaces that turn out to be equal
to the fixed spaces and move-sets of permutations.

\begin{defn}[Special vectors]\label{def:special}
  Let $\R^n$ be a euclidean vector space and let $e_1, e_2,\ldots,
  e_n$ be a fixed ordered orthonormal basis. For each subset
  $\emptyset \neq B \subset [n]$, let $\one_B$ be the sum of the basis
  vectors indexed by the numbers in $B$, i.e., $\one_B = \sum_{i\in B}
  e_i$.  We call $\one_B$ a \emph{special vector}.  When $B = \{i\}$
  only contains a single element, $\one_{\{i\}} = e_i$ is a basis
  vector.  At the other extreme, $\one_{[n]} = \one = (1, \ldots, 1)$
  is the all-$1$s vector.
\end{defn}

\begin{defn}[$F$-spaces]\label{def:f-spaces}
  For each partition $P \in \prt_n$, we define the \emph{$F$-space}
  $F_P \subset \R^n$ to be the set of all vectors whose coordinates
  are constant on the blocks of $P$.  In other words, a vector $v =
  (v_1, \ldots, v_n)$ is in $F_P$ if and only if $v_i = v_j$ for all
  $i, j$ belonging to the same block $B \in P$.
\end{defn}  

\begin{remark}[$F$-spaces]
  When $P = \{B_1, B_2, \ldots, B_k\}$, we have $F_P$ is equal to the
  span of $\one_{B_1}$, $\one_{B_2}$, \ldots, $\one_{B_k}$, and these
  special vectors form a basis for $F_P$.  In particular, $\dim(F_P) =
  \size{P}$. At the extremes, $F_{P_\zero} = \spn(\{e_1,e_2,\ldots
  e_n\}) = \R^n$ is the entire space and $F_{P_\one} = \spn(\{\one\})
  = \R\one$ is a line.  Also note that $F_P \subset F_Q$ if and only
  if $Q \leq P$ in $\prt_n$.  Thus every $F$-space contains the line
  $F_{P_\one} = \R\one$.
\end{remark}

\begin{defn}[$M$-spaces]\label{def:m-spaces}
  For each partition $P \in \prt_n$, we define the \emph{$M$-space}
  $M_P \subset \R^n$ to be the orthogonal complement of the
  corresponding $F$-space $F_P$.  In symbols, $M_P = (F_P)^\perp$ and
  $F_P \oplus M_P = \R^n$.  In terms of coordinates, $M_P$ is the set
  of vectors $v$ such that, for each block $B$ of $P$, the sum of the
  coordinates of $v$ indexed by elements of $B$ is equal to $0$.
\end{defn}

\begin{remark}[$M$-spaces]\label{rem:m-spaces}
  Since $\dim(F_P) = \size{P}$, we have $\dim(M_P) =
  n-\size{P}$. Since $F_{P_\zero} = \R^n$, $M_{P_\zero} = \{0\}$ is
  the trivial subspace.  At the other extreme, since $F_{P_\one} =
  \R\one$ is a line, $M_{P_\one} = (F_\one)^\perp$ is the
  codimension-$1$ subspace that we call $V$.  Taking orthogonal
  complements reverses containment, so $M_P \subset M_Q$ in $\R^n$ if
  and only if $P \leq Q$ in $\prt_n$.  Because every $F$-space
  contains the line $\R\one = F_{P_\one}$, every $M$-space is
  contained in the subspace $V = M_{P_\one}$.
\end{remark}

\begin{prop}[Permutations and partitions]\label{prop:perm-part}
  For each $\pi \in \Sym_n$ with partition $P(\pi) \in \prt_n$, the
  elliptic isometry $u_\pi$ acting on $\R^n$ has $\fix(u_\pi) =
  F_{P(\pi)}$ and $\mov(u_\pi) = M_{P(\pi)}$.
\end{prop}

\begin{proof}
  That a point $v$ is fixed under the action of $u_\pi$ if and only if
  $v$ is in $F_{P(\pi)}$ is clear from the definition of the action,
  which proves the assertion about the fixed space.  Since move-sets
  and fixed spaces of elliptic isometries are orthogonal complements
  \cite[Lemma~3.6]{BradyMcCammond-15}, as are $F_{P(\pi)}$ and
  $M_{P(\pi)}$ (Definition~\ref{def:m-spaces}), the assertion about
  move-sets follows from the assertion about the fixed space.
\end{proof}

\begin{remark}[Root arrangement]\label{rem:arr-affS}
  Proposition~\ref{prop:perm-part} shows that the move-set of every
  permutation $\pi \in \Sym_n$ is an $M$-space.  Moreover, it is easy
  to see that every partition $P$ of $[n]$ is induced as $P = P(\pi)$
  for some permutation $\pi$.  Thus, the move-sets of permutations are
  exactly the $M$-spaces.  It follows from Lemma~\ref{lem:red-ell}
  that these spaces are also exactly the set of root spaces for the
  root system $\Phi_n$, and so the arrangement of $M$-spaces is
  exactly the root space arrangement
  (Definition~\ref{def:move-set-arr}).
\end{remark}

The root arrangement for $\Phi_n$ should not be confused with the
better known braid arrangement.

\begin{rem}[Braid arrangement]\label{rem:braid-arr}
  The \emph{braid arrangement} in $\R^n$ is the collection of
  $\binom{n}{2}$ hyperplanes that are orthogonal to a root in
  $\Phi_n$.  Concretely, there is a hyperplane $H_{ij}$ for all
  distinct $i,j\in [n]$ that contains the vectors $v$ with $v_i =
  v_j$.  The result of taking all intersections of hyperplanes in the
  braid arrangement is a subspace arrangement that is equal to the
  collection of $F$-spaces $\{F_P \mid P \in \prt_n\}$.  By
  Remark~\ref{rem:arr-affS}, the spaces in the root arrangement
  $\arr(\Phi_n)$ are the $M$-spaces, i.e., the orthogonal complements
  of the $F$-spaces.  The hyperplanes in $\arr(\Phi_n)$ are the
  $M$-spaces $M_P$ where $P$ is a partition of size $2$, and there are
  exactly $(2^n-2)/2 = 2^{n-1}-1$ of these.  Unlike the $F$-spaces,
  the $M$-spaces are \emph{not} closed under intersection.  For
  example, in $\R^4$, the intersection $M_{\{\{1,2\},\{3,4\}\}} \cap
  M_{\{\{1,4\},\{2,3\}\}}$ is the line consisting of vectors of the
  form $(a,-a,a,-a)$ for $a\in \R$, and this is not an $M$-space.
\end{rem}

\subsection{Null partitions}\label{sec:null-parts}

In this section, we assemble the final combinatorial and
linear-algebraic objects necessary to reinterpret reflection length.

\begin{defn}[$L$-maps]\label{def:l-maps}
  Let $v$ be a vector in $\R^n$.  For each nonempty subset $B \subset
  [n]$, we define the real number $v_B$ to be the sum of the
  coordinates of $v$ in the positions indexed by the set $B$.
  Equivalently, $v_B = \langle v, \one_B\rangle$, where $\one_B$ is
  the special vector defined in Definition~\ref{def:special}.  For
  each partition $P \in \prt_n$ with $\size{P} = k$, we define a
  linear transformation $L_P\colon \R^n \to \R^k$ as follows.  When $P
  = \{B_1,\ldots, B_k\}$ we assign the names $e_{B_i}$ to an
  orthonormal basis of $\R^k$ and we define $L_P(v) = \sum_{i=1}^k
  v_{B_i} e_{B_i}$.  We call the maps of this form \emph{$L$-maps}.
\end{defn}

If we impose an order on the blocks of $P$, then in coordinates
$L_P(v) = (v_{B_1},v_{B_2},\ldots, v_{B_k})$, but note that $L_P$ is a
well defined map independent of such an ordering.  For example, if $v
= (-5,2,1,0,4,-2)$ and $P= \{\{1,3\},\{2,4,6\},\{5\}\}$ with the
blocks in that order, then $L_P(v) = (-4,0,4)$.

\begin{prop}[$M$-space membership]\label{prop:m-l}
  Let $P \in \prt_n$ be a partition of $[n]$.  Then $\ker L_P = M_P$.
\end{prop}

\begin{proof}
  Let $P = \{B_1, \ldots, B_k\}$.  The vector $v$ is in the space
  $M_P$ if and only if it is orthogonal to every vector in $F_P$.  Since
  $\{\one_{B_1}, \ldots, \one_{B_k}\}$ is a basis of $F_P$, it is
  sufficient to test whether $v$ is orthogonal to each $\one_{B_i}$,
  and this is equivalent to being in the kernel of map $L_P$.
\end{proof}

\begin{defn}[Null partitions]\label{def:null-parts}
  Let $v$ be a vector in $V = M_\one$. A nonempty subset $B\subset
  [n]$ is called a \emph{null block} of $v$ if the sum $v_B$
  (Definition~\ref{def:l-maps}) of the entries of $v$ indexed by $B$
  is equal to $0$, or equivalently if $v$ is orthogonal to the special
  vector $\one_B$ (Definition~\ref{def:special}).  A partition $P \in
  \prt_n$ is called a \emph{null partition} of $v$ if every block of
  $P$ is a null block of $v$.  Equivalently, $P$ is a null partition
  of $v$ if and only if $v$ is in the kernel of $L_P$
  (Definition~\ref{def:l-maps}) and if and only if $v \in M_P$
  (Proposition~\ref{prop:m-l}).  We write $\nprt(v)$ for the
  collection of all null partitions of $v$.  This collection is not
  empty since $v\in V$ implies $\langle v, \one \rangle = 0$, which
  means $P_\one \in \nprt(v)$.  The \emph{nullity} of $v$ is the
  maximal size of a null partition of $v$.  In symbols, $\nll(v) =
  \max\{ \size{P} \mid P \in \nprt(v)\}$.
\end{defn}

\begin{example}[Null partitions]\label{ex:null-parts}
  Consider $v = (-5,2,1,0,4,-2)$ in $V = M_\one \subset \R^6$.  The
  partition $P = \{\{1,3,5\},\{4\},\{2,6\}\}$ is a null partition of
  $v$ and its blocks are null blocks of $v$.  The set $\{1,2,3\}$ is
  not a null block of $v$ because $\langle v, \one_{\{1,2,3\}}\rangle
  = -2 \neq 0$.
\end{example}

\begin{prop}[Nullity and dimension]\label{prop:null-dim}
  For each $v \in V$, $\nll(v) + \dim_{\Phi_n}(\{v\}) = n$.
\end{prop}

\begin{proof}
  By Definition~\ref{def:root-dim}, $\dim_{\Phi_n}(\{v\})$ is the
  minimal dimension of a root space in $\arr(\Phi_n)$ that contains
  $v$.  By Remark~\ref{rem:arr-affS}, root spaces ares $M$-spaces.
  For each $P \in \prt_n$, $P \in \nprt(v)$ if and only if $v \in M_P$
  (Definition~\ref{def:null-parts}) and $\size{P} + \dim(M_P) = n$
  (Remark~\ref{rem:m-spaces}).  Thus, among the null partitions for
  $v$, the ones that maximize $|P|$ simultaneously minimize
  $\dim(M_P)$.  For a partition $P$ with these properties, the
  equation $\size{P} + \dim(M_P) = n$ becomes $\nll(v) +
  \dim_{\Phi_n}(\{v\}) = n$, as claimed.
\end{proof}

\subsection{Combinatorial formulas for statistics}

In this section, we complete the work of
Section~\ref{sec:aff-sym}. For any affine permutation $w$, we give
combinatorial expressions for the elliptic dimension $e(w)$ and the
differential dimension $d(w)$. Thus by Theorem \ref{main:dim}, we
obtain a combinatorial expression for the reflection length $\lr(w)$.

First, we consider the elliptic dimension.  Our result generalizes
D\'enes's result that the reflection length of a permutation $\pi$ is
$n - \size{\cyc(\pi)}$.

\begin{prop}[Elliptic dimension]\label{prop:e-affS}
  Let $w \in \affS_n$ be an affine permutation, with normal form $w =
  t_\lambda u_\pi$ (so that $u_\pi$ is the elliptic part of $w$).
  Then $e(w) = n - \size{\cyc(\pi)}$.
\end{prop}

\begin{proof}
 By Remark~\ref{rem:computing-e}, $e(w) = \dim(\mov(u_\pi))$.  By
 Proposition~\ref{prop:perm-part}, this is equal to
 $\dim(M_{P(\pi)})$.  As noted in Remark~\ref{rem:m-spaces}, this is
 equal to $n - \size{P(\pi)} = n - \size{\cyc(\pi)}$, as claimed.
\end{proof}

For the differential dimension we need a relative notion of nullity.

\begin{defn}[Relative nullity]\label{def:rel-null}
  Let $w = t_\lambda u_\pi$ be an affine permutation in $\affS_n$ with
  $\pi \in \Sym_n$ and $\lambda$ in the $\Z$-span of $\Phi_n$ and let
  $k = \size{P(\pi)} = \cyc(\pi)$ be the number of cycles of the
  permutation $\pi$ (Definition~\ref{def:perms}).  We call the nullity
  of the point $L_{P(\pi)}(\lambda)$ in the arrangement $\arr(\Phi_k)$
  the \emph{relative nullity of $\lambda$ modulo $\pi$} and we denote
  this number by $\nll(\lambda/\pi)$.  In combinatorial language, the
  relative nullity is the maximum size of a null partition $P$ of
  $\lambda = w(0)$ such that every cycle of $\pi = p(w)$ is contained
  in a single part of $P$.
\end{defn}

\begin{prop}[Differential dimension]\label{prop:d-affS}
  An affine permutation $w \in \affS_n$ with normal form $w =
  t_\lambda u_\pi$ has $d(w) = |\cyc(\pi)|-\nll(\lambda/\pi)$.
\end{prop}

\begin{proof}
  Let $U = \mov(u_\pi)$ be the move-set of the elliptic part $u_{\pi}$
  of $w$, and let $k = \size{P(\pi)} = \size{\cyc{\pi}}$.  By
  Remark~\ref{rem:computing-d}, the differential dimension $d(w)$ is
  equal to the dimension of the point $\lambda/U$ in the subspace
  arrangement $\arr(\Phi_n/U)$.  By Proposition~\ref{prop:perm-part},
  the move-set $U$ is equal to the $M$-space $M_{P(\pi)}$, and by
  Proposition~\ref{prop:m-l}, quotienting by the $M$-space
  $M_{P(\pi)}$ is achieved by applying the $L$-map $L_{P(\pi)}$.  In
  particular, the point $\lambda/U$ is the point $L_{P(\pi)}(\lambda)$
  in $\R^k$.
  
  Let $\overline{\imath}$ denote the block of $P(\pi)$ containing $i$.
  The image of the root $e_i-e_j$ in $\Phi_n$ under $L_{P(\pi)}$ is
  $e_{\overline{\imath}}-e_{\overline{\jmath}}$, which is either equal
  to $0$ (if $i$ and $j$ belong to the same cycle of $\pi$) or is a
  root in $\Phi_k$.  Moreover, for each root in $\Phi_k$, we can find
  a preimage in $\Phi_n$ by picking representatives.  It follows that
  the quotient arrangement $\arr(\Phi_n/U)$ is equal to the root
  arrangement $\arr(\Phi_k)$.  Thus $d(w)$ is the dimension of the
  point $L_{P(\pi)}(\lambda)$ in the arrangement $\arr(\Phi_k)$.  By
  Proposition~\ref{prop:null-dim} and Definition~\ref{def:rel-null},
  we have
  \[
   \nll(\lambda/\pi) + \dim_{\Phi_k}(\lambda/U) = k.
  \]
  As $k=|\cyc(\pi)|$, we can rewrite this as
  \[
   d(w) = |\cyc(\pi)|-\nll(\lambda/\pi),
  \]
  as claimed.
\end{proof}

Combining these propositions gives us a combinatorial formula for the
reflection length of an element in an affine symmetric group.

\begin{thm}[Formula]\label{thm:lr-affS}
  An affine permutation $w \in \affS_n$ with normal form $w =
  t_\lambda u_\pi$ has $\lr(w) = n - 2\cdot \nll(\lambda/\pi) +
  \size{\cyc(\pi)}$.
\end{thm}

\begin{proof}
  By Theorem~\ref{main:dim} and Propositions~\ref{prop:e-affS}
  and~\ref{prop:d-affS}, 
  \[
    \lr(w) = 2\cdot (\size{\cyc(\pi)} - \nll(\lambda/\pi)) + (n - \size{\cyc(\pi)}),
  \] 
  which simplifies to the expression in the statement of the theorem.
\end{proof}

\section{Future work\label{sec:future}}

\subsection{Non-euclidean Coxeter groups}

Is there a natural extension of Theorem~\ref{main:dim} to other
infinite Coxeter groups?  Because reflection length is unbounded for
all irreducible non-euclidean Coxeter groups of infinite type
\cite{Duszenko12}, there cannot be a simple formula with a bounded
number of summands representing bounded dimensions.  Thus any
extension must necessarily have a different flavor.

\subsection{Identifying reduced factorizations}

It is easy to characterize when reflection is minimal in spherical
Coxeter groups: the product $r_1 \cdots r_k$ of reflections in a
spherical Coxeter group has reflection length $k$ if and only if the
associated roots are linearly independent.  Is there a similar
criterion for affine Coxeter groups?  Because of the delicate behavior
of some of our examples in the affine symmetric group, the criterion
is likely to be complicated.

\bibliography{refs} 
\bibliographystyle{amsalpha}

\bigskip

\appendix
\section{Null partition computations}\label{sec:null-comp}

Theorem \ref{thm:lr-affS} reduces the computation of reflection length
in the affine symmetric group (which involves minimizing something
over an infinite search space) to the finite problem of computing two
combinatorial quantities: the number of cycles $|\cyc(\pi)|$ and the
relative nullity $\nll(\lambda/\pi)$.  Counting the number of cycles in
a permutation is an old problem.  In the rest of this section, we
focus on the computation of the nullity of a vector.

A first observation is that this task is computationally intractable
in general: the NP-complete problem \texttt{SubsetSum} asks, ``given a
(multi)set of integers, does it have a subset with sum $0$?''  Given
an instance $S$ of \texttt{SubsetSum}, we may reduce it to a problem
of relative nullity, as follows: create the vector $\lambda$ whose
entries consist of the entries of $S$ in some order (with correct
multiplicities) followed by the entry $- \sum_{s \in S} s$ and compute
the nullity $\nll(\lambda/\id)$.  We have that the largest partition
size is $1$ if and only if $S$ has no subset summing to $0$, answering
a \texttt{SubsetSum} instance.

Nevertheless, it is possible in practice to compute nullity by hand in
examples of reasonable size, by first identifying the null blocks, then
the minimal null blocks, then maximal null partitions and finally the
nullity.  Throughout this section, $v \in L(\Phi_n)$ is always an
integer vector with coordinate sum $0$ and we illustrate our
procedures using the specific vector $v_0 = (-3,-2,-2,-1,1,2,5) \in
\R^7$.  To compute the null blocks of a vector we divide its
coordinates according to their sign.

\begin{defn}[Positive and negative weights]\label{def:weights}
  Let $X = \{ i \mid v_i>0\}$, $Y = \{i \mid v_i <0 \}$ and $Z = \{i
  \mid v_i = 0\}$. Of course $X \cup Y \cup Z = [n]$.  Recall that for
  each $B \subset [n]$, $v_B = \langle v,\one_B\rangle$ is the sum of
  the coordinates indexed by the numbers in $B$ and that saying $B$ is
  a null block means $v_B = 0$.  For each subset $B \subset [n]$, let
  $B_+ = B \cap X$, $B_- = B \cap Y$ and $B_0 = B \cap Z$. We call
  $v_{B_+}$ the \emph{positive weight of $B$} and $v_{B_-}$ the
  \emph{negative weight of $B$} and note that $B$ is a null block if
  and only if $v_{B_+} + v_{B_-} = 0$.
\end{defn}

To find the null partitions of maximal size, it suffices to restrict
our attention to those constructed out of null blocks that are
minimal under inclusion.

\begin{lem}[Maximal null partitions]
  Let $P$ be a null partition of $v\in V$.  If some block $B \in P$ is
  not minimal among the null blocks of $v$, then $P$ is not maximal
  among the null partitions of $v$.  In particular, maximal null
  partitions are constructed out of minimal null blocks.
\end{lem}

\begin{proof}
  Let $B' \subset B$ be a proper nonempty subset of $B$ where both are
  null blocks of $v$. Since the sum of the coordinates indexed by $B$
  and by $B'$ add to $0$, this means that the coordinates in $B'' = B
  \setminus B'$ also add to $0$.  Thus $B''$ is a null block of $v$
  and the refined partition that replaces the block $B$ with the pair
  of blocks $B'$ and $B''$ is a new null partition of $v$ of strictly
  larger size.
\end{proof}

We now describe a handy way to identify minimal null blocks.

\begin{defn}[Profiles]\label{def:profile}
  We create two lists, $\Pos(v)$ and $\Neg(v)$.  Each list has length
  equal to $v_X = -v_Y$ and each entry is initially the empty set.
  For each nonempty subset of $B_+ \subset X$ we add the set $B_+$ to
  the collection of sets in the $i$-th position of the first list
  where $i = v_{B_+}$ is the (positive) weight of $B_+$.  The result
  is $\Pos(v)$, the \emph{positive profile of $v$}.  Similarly, for
  each nonempty subset of $B_- \subset Y$ we add the set $B_-$ to the
  collection of sets in the $i$-th position in the second list, where
  $i = -v_{B_-}$ is the absolute value of the negative weight of
  $B_-$.  The result is $\Neg(v)$, the \emph{negative profile of $v$}.
  Let $\Pos(v,i)$ be the $i$-th entry in the list $\Pos(v)$ and let
  $\Neg(v,i)$ be the $i$-th entry in the list $\Neg(v)$.
\end{defn}

\begin{rem}[Null blocks]
  If $B$ is a null block of positive weight $i$ then $B_+$ is in
  $\Pos(v,i)$, $B_-$ is in $\Neg(v,i)$ and $B_0 \subset Z$.
  Conversely, for every choice of $B_+ \in \Pos(v,i)$, $B_- \in
  \Neg(v,i)$ and $B_0 \subset Z$, the set $B = B_+ \cup B_- \cup B_0$
  is a null block of positive weight $i$.  Since we are ultimately
  interested in minimal null blocks, we ignore the subset of $Z$ and
  call a null block with $B \cap Z = \emptyset$ a \emph{basic null
    block}.  To create a list of all basic null blocks for $v$ we
  simply take the ``dot product'' of the positive and negative
  profiles of $v$.  Concretely $\Basic(v,i) = \{ B_+ \cup B_- \mid B_+
  \in \Pos(v,i), B_- \in \Neg(v)\}$ is the collection of all
  \emph{basic null blocks of positive weight $i$} and if $k$ is the
  positive weight of $[n]$, then $\Basic(v) =
  (\Basic(v,1),\ldots,\Basic(v,k))$ is a list of all basic null
  blocks, filtered by their positive weight.
\end{rem}

\begin{figure}
  \begin{tikzpicture}[scale=.9]
    \draw (0,0) grid (8,8);
    \draw[line width=2] (0,8)--(0,0)--(8,0);
    \draw[dashed] (0,0)--(8,8);
    \foreach \x in {1,...,7}{
      \draw (0,\x) node[Square] {$\x$};
      \draw (\x,0) node[Square] {$\x$};
    }
    \node[Circle] at (1,1) {$1$};
    \node[Circle] at (2,2) {$2$};
    \node[Circle] at (3,3) {$3$};
    \node[Circle] at (4,4) {$0$};
    \node[Circle] at (5,5) {$3$};
    \node[Circle] at (6,6) {$2$};
    \node[Circle] at (7,7) {$1$};
    \begin{scope}[left, color=red]
      \node at (-.5,1) {$4$};
      \node at (-.5,2) {$2,3$};
      \node at (-.5,3) {$1,24,34$};
      \node at (-.5,4) {$14,23$};
      \node at (-.5,5) {$12,13,234$};
      \node at (-.5,6) {$124,134$};
      \node at (-.5,7) {$123$};
    \end{scope}
    \begin{scope}[below, color=blue]
      \node at (1,-.5) {$5$};
      \node at (2,-.5) {$6$};
      \node at (3,-.5) {$56$};
      \node at (4,-.5) {};
      \node at (5,-.5) {$7$};
      \node at (6,-.5) {$57$};
      \node at (7,-.5) {$67$};
    \end{scope}
  \end{tikzpicture}
  \caption{A graphical representation of the positive and negative
    profiles and the basic null blocks of $v_0$.\label{fig:profile}}
\end{figure}
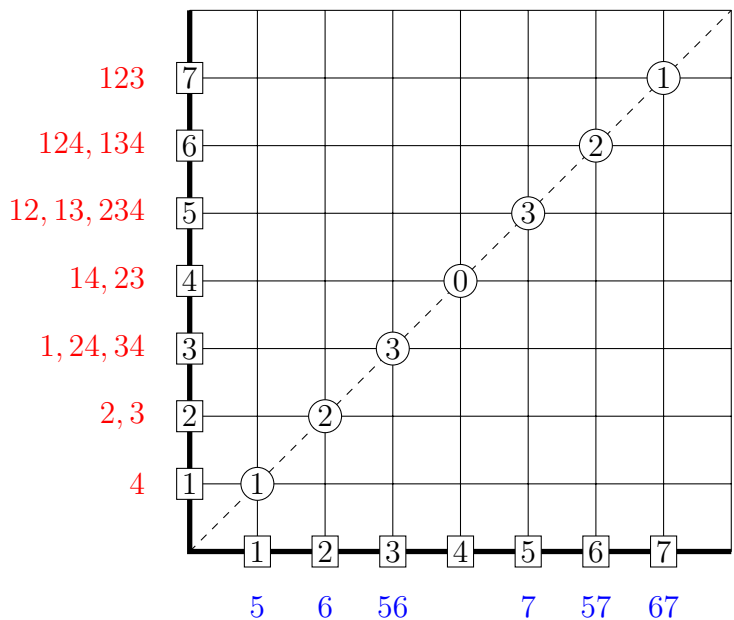

\begin{example}[Profiles]\label{ex:profile}
  For the vector $v_0 = (-3,-2,-2,-1,1,2,5)$, $X = \{5,6,7\}$, $Y =
  \{1,2,3,4\}$, $Z = \{\}$ and its positive weight is $8$.  The
  positive profile $\Pos(v_0)$, in simplified notation, is shown
  beneath the $x$-axis in Figure~\ref{fig:profile}.  Each proper
  nonempty subset of $X$ is shown without internal commas and
  surrounding braces and placed beneath the $x$-axis at a location
  indicating its (positive) weight.  The full set $X$ is included as a
  set in $\Pos(v_0,8)$ but it is not displayed.  Similarly, the
  simplified negative profile of $v_0$ is shown to the left of the
  $y$-axis.  The number on the diagonal at location $(i,i)$ is the
  product $\size{\Pos(v_0,i)} \cdot \size{\Neg(v_0,i)} =
  \size{\Basic(v_0,i)}$, the number of basic null blocks of positive
  weight $i$.  The evident symmetry is not accidental since the
  complement of a proper null block is a proper null block.  From this
  data we find that $v_0$ has fourteen proper null blocks: $\{4,5\}$,
  $\{2,6\}$, $\{3,6\}$, $\{1,5,6\}$, $\{2,4,5,6\}$, $\{3,4,5,6\}$,
  $\{1,2,7\}$, $\{1,3,7\}$, $\{2,3,4,7\}$, $\{ 1,2,4,5,7\}$,
  $\{1,3,4,5,7\}$ and $\{ 1,2,3,6,7\}$.
\end{example}

\begin{rem}[Finding minimal null blocks]\label{def:min-null}
  If $B$ is a basic null block of $v$ that is not minimal, then it
  must contain a basic null block of strictly smaller positive weight.
  This leads to an algorithm that quickly identifies the minimal null
  blocks from the filtered list $\Basic(v)$.  Let $L$ be a copy of
  $\Basic(v)$ and we proceed to modify $L$ working from left to right.
  The basic null blocks, if any, in the first entry of $L$ are
  necessarily minimal null blocks, so we count them as confirmed and
  remove from $L$ any basic null blocks (to the right) that contain
  one of these as a subset.  At this point, the basic null blocks
  remaining in the second entry are necessarily minimal, we count them
  as confirmed and remove basic null blocks that contain one of them
  as a subset.  And so on.  In the end the only blocks that remain are
  minimal null blocks.
\end{rem}

\begin{figure}
  $\begin{array}{c}
    \mid (45)(26, 36)(156, 2456, 3456)(127, 137, 2347)(12457, 13457) (12367)\\
    \leadsto (45) \mid (26, 36)(156)(127, 137,  2347)(12367)\\
    \leadsto (45)(26, 36) \mid (156)(127, 137, 2347)\\
    \leadsto (45)(26, 36)(156) \mid (127, 137, 2347)\\
    \leadsto (45)(26, 36)(156)(127, 137, 2347) \mid\\
  \end{array}$
  \medskip
 \caption{Minimual null block algorithm applied to $v_0$.
\label{fig:min-null}}
\end{figure}

\begin{example}[Finding minimal null blocks]\label{ex:min-null}
  Figure~\ref{fig:min-null} shows the progress of the minimum null
  block algorithm described in Definition~\ref{def:min-null} when
  applied to $v_0$, presented in a simplified notation.  As in
  Example~\ref{ex:profile} blocks are listed without commas and
  braces, commas separate distinct blocks. Parentheses are used to
  indicate blocks with the same positive weight and these sets are
  listed from left to right according to their common positive weight.
  The vertical bar separates the confirmed minimal null blocks to the
  left from the unprocessed basic null blocks to the right.  At the
  first step the null block $\{4,5\}$ is seen to be minimal and we
  remove $\{2,4,5,6\}$, $\{3,4,5,6\}$, $\{1,2,4,5,7\}$ and
  $\{1,3,4,5,7\}$ from the list.  In the second step both $\{2,6\}$
  and $\{3,6\}$ are confirmed as minimal and $\{1,2,3,6,7\}$ is
  removed from the list.  The third step confirms $\{1,5,6\}$ is
  minimal and the fourth step confirms $\{1,2,7\}$, $\{1,3,7\}$ and
  $\{2,3,4,7\}$ are minimal.  Thus $v_0$ has exactly $7$ minimal null
  blocks.
\end{example}

Once the minimal null blocks have been identified, it remains to
identify the maximal null partitions, and one way to do this is to
build a flag simplicial complex.

\begin{defn}[Null partition complex]\label{def:complex}
  Let $\Gamma$ be a graph with vertices indexed by the minimal null
  blocks of $v$ and with an edge connecting two vertices if and only
  if their minimal null blocks are disjoint.  Every graph can be
  turned into a flag simplicial complex by adding a simplex spanning a
  subset of vertices if and only if the $1$-skeleton of that simplex
  is already visible in the graph.  In other words, there is a simplex
  in the complex for every complete subgraph of the graph.  Let
  $\ncplx(v)$ be the flag simplicial complex built from $\Gamma$.  We
  call $\ncplx(v)$ the \emph{null partition complex of $v$}.  Since
  complements of null blocks are null blocks and every null block is a
  union of pairwise disjoint minimal null blocks, a null partition of
  $v$ is maximal if and only if it corresponds to a maximal simplex in
  $\ncplx(v)$.  This direct correspondence means that the nullity of a
  vector is one more than the dimension of its null partition complex.
  In symbols, $\nll(v) = \dim(\ncplx(v)) + 1$.
\end{defn}

\begin{figure}
  \begin{tikzpicture}
    \draw[Poly] (0,0)--(30:2)--(-30:2)--cycle;
    \draw[Poly] (0,0)--(150:2)--(-150:2)--cycle;
    \draw[Poly] (-1,-2)--(1,-2);
    \node[Rounded] at (0,0) {$45$};
    \node[Rounded] at (150:2) {$127$};
    \node[Rounded] at (-150:2) {$36$};
    \node[Rounded] at (30:2) {$137$};
    \node[Rounded] at (-30:2) {$26$};    
    \node[Rounded] at (-1,-2) {$156$};    
    \node[Rounded] at (1,-2) {$2347$};
  \end{tikzpicture}
  \caption{Null partition complex for $v_0$.\label{fig:complex}}
\end{figure}
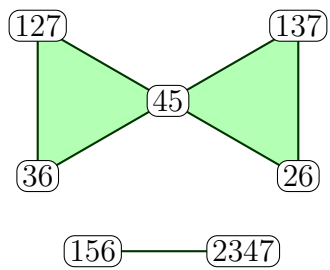

\begin{example}[Null partition complex]\label{ex:complex}
  The null partition complex of $v_0$ is shown in
  Figure~\ref{fig:complex}.  It has $7$ vertices, $7$ edges and $2$
  triangles.  There are three maximal simplices: two triangles and an
  edge and these correspond to maximal null partitions
  $\{\{1,2,7\},\{3,6\},\{4,5\}\}$, $\{\{1,3,7\},\{2,6\},\{4,5\}\}$ and
  $\{\{1,5,6\},\{2,3,4,7\}\}$, respectively.  The dimension of
  $\ncplx(v_0)$ is $2$ and its nullity is $3$.
\end{example}

\end{document}